\documentclass{amsart}
\usepackage{enumerate}
\usepackage{amssymb,latexsym}
\usepackage{graphicx}
\usepackage{verbatim}
\usepackage{xcolor} 
\usepackage{hyperref}
\hypersetup{
    linktoc=all,     
    linkcolor=black,  
}

\newtheorem{theorem}{Theorem}
\newtheorem{corollary}[theorem]{Corollary}
\newtheorem{lemma}[theorem]{Lemma}
\newtheorem{proposition}[theorem]{Proposition}

\newtheorem{example}[theorem]{Example}
\newtheorem{remark}[theorem]{Remark}
\newtheorem{remarks}[theorem]{Remarks}

\newcommand{\R}{\mathbb{R}}

\begin{document}
\title{Ribaucour partial tubes and hypersurfaces of Enneper type}

\maketitle
\begin{center}
\author{Sergio Chion  
       \footnote{This research is a result of the activity developed within the framework of the Programme in Support of Excellence Groups of the Regi\'on de Murcia, Spain, by Fundaci\'on S\'eneca, Science and Technology Agency of the Regi\'on de Murcia. Sergio Chion was partially supported by Fundaci\'on S\'eneca project 19901/GERM/15, Spain.}
\and
        Ruy Tojeiro
        \footnote{The second author is partially 
supported by Fapesp grant 2016/23746-6 and 
CNPq grant 303002/2017-4.}}
\end{center}
\date{}

\begin{abstract}
In this article we introduce the notion of a Ribaucour partial tube
and use it to derive several applications. These are based on a characterization of  Ribaucour partial tubes as the immersions of a product of two manifolds into a space form such that the distributions given by the tangent spaces of the factors are orthogonal to each other  with respect to the induced metric, are invariant under all shape operators, and one of them is spherical. Our first application is a classification of all hypersurfaces with dimension at least three of a space form that carry a spherical foliation of codimension one, extending previous results by Dajczer, Rovenski and the second author for the totally geodesic case. We proceed to prove a general decomposition theorem for immersions of product manifolds, which extends several related results. Other main applications concern the class of hypersurfaces of $\R^{n+1}$ that are of  Enneper type, that is, hypersurfaces that
 carry a family of lines of curvature, correspondent to a simple principal curvature, whose orthogonal $(n-1)$-dimensional distribution is integrable and whose leaves are contained in hyperspheres or affine hyperplanes of $\mathbb{R}^{n+1}$. 
We show how Ribaucour partial tubes in the sphere can be used to describe all $n$-dimensional hypersurfaces of Enneper type for which the leaves of the 
$(n-1)$-dimensional distribution are contained in affine hyperplanes of $\mathbb{R}^{n+1}$,
and then show  how a general hypersurface of Enneper type can be constructed in terms of a hypersurface in the latter class. We give an explicit description of some special  hypersurfaces of Enneper type, among which are natural generalizations of the so called Joachimsthal surfaces.
\end{abstract}

\noindent \emph{2020 Mathematics Subject Classification:} 53 B25.\vspace{2ex}

\noindent \emph{Key words and phrases:} {\small {\em Ribaucour partial tube,  spherical foliation, \\
 hypersurface of Enneper-type, conformal polar metric.}}

\section{Introduction}

In  \cite{dt3} and \cite{drt}, the authors studied the problem of determining the  
hypersurfaces  of dimension at least three of a space form that carry  
a totally geodesic foliation of codimension one. The initial motivation of this 
work was to investigate the similar problem that one can pose by assuming the 
foliation to be spherical instead of being totally geodesic.  
A foliation being spherical means that each leaf is an umbilical submanifold 
whose mean curvature vector field is parallel with respect to its normal connection. 

 In the solution of the problem addressed in \cite{dt3} and  \cite{drt},
one main example of a hypersurface of a space form that carries 
a totally geodesic foliation of codimension one is a partial tube over a 
smooth regular curve. Partial tubes are submanifolds that are generated by 
starting with any submanifold whose normal bundle has a parallel and flat subbundle,  
taking a submanifold in the fiber of that subbundle at a given point, and then 
parallel transporting it along the former submanifold with respect to its normal connection. They are characterized by the property that the tangent spaces
to the submanifold that is parallel transported along the 
starting submanifold give rise to a totally geodesic distribution that is invariant under
all of its shape operators.

  In this paper we introduce a class of submanifolds that play the role of 
partial tubes in our context, and which are the basis for some of our main results.
They are immersions of product manifolds with two factors that contain partial 
tubes as special cases and whose construction is based on the extension of 
the Ribaucour transformation for submanifolds developed in \cite{dt} and \cite{dt2}, 
so we call them \emph{Ribaucour partial tubes} (see Section $3$ for the precise definition). They extend the notion of an $\mathcal{N}$-Ribaucour transform
defined in \cite{dft}.

It is a basic result of this paper that Ribaucour partial tubes are precisely  the immersions of product manifolds with two factors such that the distributions given by the tangent spaces of the factors are orthogonal to each other with respect to the induced metric, are invariant under all shape operators, and one of them is spherical 
(Theorem \ref{rpt}). This immediately implies that any submanifold of 
a space form that carries a spherical distribution that is invariant under all shape 
operators and whose orthogonal distribution is integrable is locally a Ribaucour 
partial tube (Corollary \ref{cor:sphdist}). In particular, this yields an explicit 
description of all surfaces with flat normal bundle of a space form such that the 
lines of curvature of one the two families have constant geodesic curvature 
(Corollary \ref{cor:surfaces}). Another immediate application is a description of all foliations of a space form by spherical submanifolds whose orthogonal distributions are integrable, in particular of all foliations of a space form whose leaves are spherical hypersurfaces (Corollary \ref{rptdiffeo}).

In the hypersurface case, we show that the condition of being invariant by the shape 
operator is automatically satisfied if the rank of the distribution is greater than 
half of the dimension of the hypersurface. In particular, it follows that any hypersurface 
with dimension $\geq 3$ of a space form that carries a spherical foliation of codimension 
one is locally a Ribaucour partial tube over a curve, giving a complete answer to the 
problem that was one of the initial motivations of this paper and leading to a rather 
explicit description of such hypersurfaces (Corollary \ref{cor:hyp}). 

Hypersurfaces of $\mathbb{R}^{n+1}$ that are Ribaucour partial tubes over curves can be also characterized by the fact that they carry a family of lines of curvature, correspondent to a simple principal curvature, whose orthogonal $(n-1)$-dimensional distribution is integrable and whose leaves are contained in hyperspheres of $\mathbb{R}^{n+1}$ that intersect the hypersurface orthogonally.  This has led us to investigate the more general class of hypersurfaces for which such leaves are contained in hyperspheres that do not necessarily intersect the hypersurface orthogonally. For $n=2$, this reduces to studying the class of surfaces with spherical lines of curvature correspondent to one of the principal curvatures, which was widely investigated by many geometers since the second half of the nineteenth century. The interest in such surfaces has been renewed in connection with the construction of immersed constant mean curvature tori in Euclidean three space by Wente \cite{we} and others. In \cite{we}, surfaces with spherical lines of curvature correspondent to one of the principal curvatures were called \emph{surfaces of Enneper type}. Accordingly, we say that a  hypersurface of $\R^{n+1}$ is of Enneper type if it carries a family of lines of curvature, correspondent to a simple principal curvature, whose orthogonal $(n-1)$-dimensional distribution is integrable and whose leaves are contained in hyperspheres or affine hyperplanes of $\mathbb{R}^{n+1}$. 

Our approach to studying hypersurfaces of  Enneper type was inspired by that presented in Bianchi's book \cite{bi} for the case $n=2$. First we show how Ribaucour partial tubes in $\mathbb{S}^{n}$ can be used to describe all hypersurfaces of Enneper type for which the leaves of the 
$(n-1)$-dimensional distribution are contained in affine hyperplanes of $\mathbb{R}^{n+1}$
(Theorem \ref{thm:sphleaves}). For $n=2$, these correspond to surfaces with planar lines of curvature associated with one of the principal curvatures. Then we prove that any  hypersurface of Enneper type can be constructed in terms of a hypersurface in the latter class (Theorem \ref{thm:class}). For that, we first show how to parametrize any hypersurface of Enneper type in $\mathbb{R}^{n+1}$ in terms of its Gauss map and a triple $(\gamma, \alpha, \beta)$, where $\gamma\colon I\to \mathbb{R}^{n+1}$ is a smooth curve and $\alpha, \beta\in C^{\infty}(I)$ (Theorem \ref{prop:sphleaves0}). Then we determine all the triples $(\bar\gamma, \bar\alpha, \bar\beta)$ that give rise to hypersurfaces of Enneper type in $\mathbb{R}^{n+1}$ with the same Gauss map as a given one (Proposition \ref{prop:enneperfamily}). It turns out that, among them, there always exists a hypersurface for which the hyperspheres containing the leaves of the $(n-1)$-dimensional distribution all pass through a common point. An inversion with respect to a hypersphere centered at that point then maps such hyperspheres into affine hyperplanes, and hence, maps the hypersurface into a hypersurface of Enneper type whose leaves are contained in affine hyperplanes of $\mathbb{R}^{n+1}$.

We also give an explicit description of hypersurfaces of Enneper type in $\mathbb{R}^{n+1}$ for which the leaves of the $(n-1)$-dimensional distribution are contained either in concentric hyperspheres, parallel affine hyperplanes or affine hyperplanes that intersect along a common affine $(n-1)$-dimensional subspace (Theorem \ref{thm:joach}). Surfaces in $\R^3$ with the last of these properties  are classically known as Joachimsthal surfaces, and our result yields a new description of these surfaces (Corollary \ref{cor:joach}). It was recently shown in \cite{st} that this last property is also shared by the so called cyclic conformally flat hypersurfaces of $\R^4$ with three distinct principal curvatures.

   In the last part of the article, the aforementioned characterization of Ribaucour partial tubes is  applied to 
the program of investigating the geometry of an isometric immersion with high codimension 
by trying to ``decompose it"  into simpler ``components". This program has similar 
counterparts in many branches of mathematics,  with the decomposition of an integer into 
prime factors as its most basic example. In differential geometry, from an intrinsic point 
of view it has led to several de Rham-type theorems, which provide conditions under which a certain Riemannian manifold is (locally or globally) isometric to a product manifold whose metric is of a certain type. 

 Here we first derive such a de Rham-type theorem that gives  conditions for a
Riemannian manifold to be locally isometric to a product manifold whose metric is
 conformal to a polar metric (Theorem \ref{t:carspherical}). Recall that
a metric $g$ on a product manifold $M = \prod_{i=0}^r M_i$ is said to be 
\emph{polar} if there exist a metric $g_0$ on $M_0$ and, for each $1 \leq a \leq r$, 
a family of metrics on $M_a$ smoothly parametrized by $M_0$, such that
$$
g = \pi_0^*g_0 + \sum_{a=1}^r \pi_a^*(g_a \circ \pi_0),	
$$
where $\pi_i\colon M \to M_i$ is the projection for $0\leq i\leq r$. 
Polar metrics include as special cases the warped product of 
metrics  $g_0,\ldots,g_r$ on $M_0,\ldots,M_r$, respectively, with smooth 
warping functions $\rho_a\colon M_0\to\R_+$, 
$1\leq a\leq r$, that is,  metrics given by 
$$
g=\pi_0^*g_0+\sum_{a=1}^r(\rho_a\circ\pi_{0})^2\pi_a^*g_a,
$$
in particular the Riemannian product of $g_0,\ldots,g_r$, for which
the warping functions $\rho_a$, $1\leq a\leq r$, are 
identically one. Warped (respectively, Riemannian) product metrics 
correspond to polar metrics for which all metrics $g_a(x_0)$ on $M_a$, 
$1\leq a\leq r$, $x_0\in M_0$, are  homothetical (respectively, isometric) 
to a fixed Riemannian metric. Our result extends previous results 
in \cite{hiepko}, \cite{t1} and \cite{t3} for warped product metrics, metrics
that are conformal to Riemannian and warped product metrics, and polar 
metrics, respectively.

From an extrinsic point of view, several decomposition theorems for immersions 
of product manifolds have been obtained under the assumption that the tangent 
spaces to the factors are invariant by all shape operators, starting from Moore's 
basic result characterizing extrinsic products of immersions among isometric 
immersions of Riemannian product manifolds that satisfy that condition. 
Moore's theorem has been generalized in \cite{nol}, \cite{t2}, \cite{t3} and 
\cite{t4} for product manifolds endowed with more general types of metrics, 
namely, warped product metrics, metrics that are
conformal to Riemannian product and warped product metrics, and polar metrics, 
respectively. Here we use the notion of a Ribaucour partial tube to provide a 
further generalization for the fairly general class of metrics that are conformal 
to polar metrics. Namely, we give a complete description of all conformal immersions 
of a product manifold endowed with a polar metric under 
the assumption that the tangent spaces of the factors are invariant by all shape 
operators (Theorem \ref{thm:decomp}).

\section{The Ribaucour transformation}

This section is devoted to review some basic facts on the Ribaucour transformation. 
For further details  we refer to \cite{dt} and \cite{dt2}.

Let $f\colon M^n \to \R^m$ be an isometric immersion of a Riemannian manifold $M^n$. 
We denote by $\mathcal{S}(M)$ the module of symmetric sections of the
vector bundle of endomorphisms of $TM$, that is, those elements of $\Gamma(\text{End}(TM))$
such that $\left<\Phi X, Y\right> = \left<X, \Phi Y\right>$ for all $X, Y\in \mathfrak{X}(M)$.

 A map $\mathcal{F}\colon M^n \to \R^m$ is called a \emph{Combescure transform} of 
an isometric immersion  $f\colon M^n \to \R^m$ determined by $\Phi \in \mathcal{S}(M)$ 
when $\mathcal{F}_* = f_* \circ \Phi$.  This condition forces $\Phi$ to satisfy the
Codazzi equation
$$(\nabla_X \Phi) Y = (\nabla_Y \Phi) X$$
and to commute with the second fundamental form of $f$, in the sense that
$$\alpha(X,\Phi Y) = \alpha(\Phi X, Y)$$
for all $X, Y\in \mathfrak{X}(M)$. Conversely, if $M^n$ is simply connected then any $\Phi \in \mathcal{S}(M)$ satisfying these two conditions determines a Combescure transform 
$\mathcal{F}\colon M^n \to \R^m$ of $f\colon M^n \to \R^m$ such that 
$\mathcal{F}_* = f_* \circ \Phi$.  
 
 For any Combescure transform $\mathcal{F}\colon M^n \to \R^m$ of an isometric immersion  
$f\colon M^n \to \R^m$ of a simply connected Riemannian manifold, there exist $\varphi \in C^\infty(M)$ and $\beta \in \Gamma(N_fM)$ satisfying
\begin{equation}\label{e:combescuresolution}
\alpha(\text{grad}\, \varphi, X) + \nabla_X^\perp \beta = 0
\end{equation}
such that 
\begin{equation}\label{e:combescuredetermination}
\mathcal{F} = f_* \text{grad}\, \varphi + \beta \quad \text{and} \quad \Phi 
= \text{Hess}\,\varphi - A_\beta.
\end{equation}
Conversely, any solution $(\varphi, \beta)$ of \eqref{e:combescuresolution} determines a 
Combescure transform $\mathcal{F}$ of $f$ defined by \eqref{e:combescuredetermination}.  
We denote by $\mathcal{D}(f)$ the space of solutions  $(\varphi, \beta)$ of \eqref{e:combescuresolution}.  

Given an isometric immersion $f\colon M^n \to \R^m$, an immersion 
$\tilde{f}\colon M^n \to \R^m$ is said to be a \emph{Ribaucour transform} of $f$ when 
$||f - \tilde{f}|| \neq 0$ everywhere and there exists a triple $(\mathcal{P},D,\delta)$, 
with $\mathcal{P}\colon f^*T\R^m \to \tilde{f}^*T\R^m$ a vector bundle isometry, 
$D \in \mathcal{S}(M)$ and $\delta \in \Gamma(f^*T\R^m)$ nowhere vanishing, such that
\begin{enumerate}
\item[(i)] $\mathcal{P}Z - Z = \left<Z,\delta\right>(f - \tilde{f})$
\item[(ii)] $\tilde{f}_* = \mathcal{P} f_* D$. 
\end{enumerate}
Geometrically, $f$ and $\tilde{f}$ envelop a common  congruence of $n$-dimensional spheres, 
with $\mathcal{P}(x)$, $x\in M^n$, being the reflection with respect to the hyperplane
orthogonal to $\tilde{f}(x)-f(x)$, $x\in M^n$. The requirement that the tensor $D$ be symmetric 
implies that the shape operators $A^f_\xi$ and $A^{\tilde f}_{\mathcal{P}\xi}$ with respect to 
corresponding normal directions commute for every $\xi\in \Gamma(N_fM)$. 
For surfaces in $\mathbb{R}^3$, this is equivalent to requiring $f$ and $\tilde f$ to share 
the same lines of curvature. 

A nice feature of the Ribaucour transformation is that all Ribaucour transforms of a given 
isometric immersion $f\colon M^n \to \R^m$ can be explicitly parametrized as follows in terms of 
$f$ and the pairs $(\varphi, \beta) \in \mathcal{D}(f)$.

\begin{theorem}[\cite{dt2}] \label{t:ribcharac}
Let $f\colon M^n \to \R^m$ be an isometric immersion of a simply connected Riemannian manifold and let 
$\tilde{f}\colon M^n \to \R^m$ be a Ribaucour transform of $f$ with data $(\mathcal{P},D,\delta)$. 
Then, there exists $(\varphi, \beta) \in \mathcal{D}(f)$ such that 
\begin{equation}\label{e:formularib}
\tilde{f} = f - 2\nu\varphi\mathcal{F},	
\end{equation}
where $\mathcal{F} = f_*\text{grad}\, \varphi + \beta $ is the Combescure transform determined by 
$(\varphi,\beta)$ and $\nu = \left<\mathcal{F},\mathcal{F}\right>^{-1}$.  Moreover,
$$
\mathcal{P} = I - 2\nu\mathcal{F}^*\mathcal{F}, 
\quad D= I - 2\nu \varphi \Phi \quad \text{and} \quad \delta = -\varphi^{-1}\mathcal{F}, 
$$	
where $\Phi = \text{Hess}\,\varphi - A_\beta$ and $\mathcal{F}^* \in \Gamma(f^*T\R^m)^*)$ is defined 
by $\mathcal{F}^*Z = \left<\mathcal{F},Z\right>$.

Conversely, given $(\varphi,\beta) \in \mathcal{D}(f)$ and an open subset $U \subset M^n$ where $\varphi$ and $\mathcal{F}= f_*\text{grad}\, \varphi + \beta$ are nowhere vanishing and $D$ is invertible, then $\tilde{f}\colon U \to \R^m$ given by \eqref{e:formularib} is a Ribaucour transform of $f|_U$.  
\end{theorem}

The Ribaucour transform determined by  $(\varphi,\beta)$ is denoted by $\mathcal{R}_{(\varphi,\beta)}f$. 

\begin{example}\label{ribs}\emph{
\noindent $(i)$ Given a point $P_0\in \R^{m}$ and $r>0$, set 
$2\varphi_1=\|f-P_0\|^2-r^2$ and $\beta_1=(f-P_0)_{N_fM}$.
Then $\mathcal{F}=f-P_0$, $\Phi=I$, and
$$
\tilde f=\mathcal{R}_{(\varphi_1,\beta_1)}(f)=P_0+r^2\|f-P_0\|^{-2}(f-P_0)
$$
is the composition of $f$ with an inversion with respect to
the sphere of radius $r$ centered at $P_0$.  \\
\noindent $(ii)$ Given a parallel normal vector field $\xi$,
define $(\varphi_2,\beta_2)$ by $2\varphi_2=\|\xi\|^2$ and $\beta_2=-\xi$.
Then $\mathcal{F}=-\xi$
and
$$
\tilde f=\mathcal{R}_{(\varphi_2,\beta_2)}(f)=f+\xi
$$
is the parallel translation  of $f$ by $\xi$. 
}\end{example}

The Ribaucour transformation can be easily extended for submanifolds of any space form 
$\mathbb{Q}_c^{m}$. Namely, an immersion $\tilde{f}\colon\, M^n\to\mathbb{Q}_c^{m}$
is a Ribaucour transform of an isometric immersion $f\colon M^n\to \mathbb{Q}_c^{m}$ with data $(\mathcal{P},D,\delta)$ if $\tilde{F}:=i\circ \tilde{f}\colon\,M^n\to\mathbb{R}_{\epsilon}^{m+1}$, where $i\colon \mathbb{Q}_c^{m}\to \mathbb{R}_{\epsilon}^{m+1}$ is the umbilical inclusion, is a Ribaucour transform of $F=i\circ f$ with data 
$(\hat{\mathcal{P}},D,\hat{\delta})$, where $\hat{\delta}=\delta-cF$ and $\hat{\mathcal{P}}\colon\,F^*\mathbb{R}_{\epsilon}^{m+1}\to\tilde{F}^*\mathbb{R}_{\epsilon}^{m+1}$ is the extension of $\mathcal{P}$ defined by setting $\hat{\mathcal{P}}(F)=\tilde{F}$. 
In this setting, Theorem \ref{t:ribcharac} reads as follows.

\begin{theorem}[\cite{dt2}]\label{pr:rcsc} 
Let $f\colon\,M^n\to\mathbb{Q}_c^{m}$
be an isometric immersion of a simply connected Riemannian manifold and let
$\tilde{f}\colon\, M^n\to \mathbb{Q}_c^{m}$ be a Ribaucour transform of $f$ with data
$(\mathcal{P},D,\delta)$. Then there exists
$(\varphi,\beta)\in \mathcal{D}(f)$ such~that
\begin{equation}\label{eq:rb2}
\tilde{F}= F - 2\nu\varphi\mathcal{G},
\end{equation}
where $\mathcal{G} =F_*\text{grad}\,\varphi + \beta +c\varphi F$ and
$\nu=\left<\mathcal{G}, \mathcal{G}\right>^{-1}$.
Moreover,
\begin{equation}\label{eq:pdo2}
\hat{\mathcal{P}}=I - 2\nu\mathcal{G}\mathcal{G}^*,\;\;\;
D=I - 2\nu\varphi(\text{Hess}\,\varphi +c\varphi I - A_\beta)
\;\;\;\mbox{and}\;\;\;\hat{\delta}
=-\varphi^{-1}\mathcal{G}.
\end{equation}
Conversely, given $(\varphi,\beta)\in \mathcal{D}(f)$ and an open subset $U\subset M^n$
where $\varphi\nu\neq 0$ and the tensor $D$ given by {\em (\ref{eq:pdo2})} is
invertible,  let $\tilde{F}\colon\,U\to\mathbb{R}_{\epsilon}^{m+1}$ be defined by
{\em (\ref{eq:rb2})}. Then $\tilde{F}=i\circ \tilde{f}$, where
$\tilde{f}$ is a Ribaucour transform of $f$.
\end{theorem}

\section{Ribaucour partial tubes}

In this section we introduce the concept of a Ribaucour partial tube, 
on which some of the main results of this article are based.

Let $f_1\colon M_1 \to \R^m$ be an isometric immersion along which there is
an orthonormal set $\{\xi_1,\ldots,\xi_k\}$ of  normal vector fields that are 
parallel in the normal connection.
The subbundle ${\mathcal L}=\text{span}\,\{\xi_1,\ldots,\xi_k\}$  
of $N_{f_1}M_1$ is thus parallel and flat. Hence the map 
$\Psi\colon M_1\times\R^k\to {\mathcal L}$, defined by
$$
\Psi_{x_1}(y)=\Psi(x_1,y)=\sum_{i=1}^ky_i\xi_i(x_1)
$$ 
for all $x_1\in M_1$ and $y=(y_1,\ldots, y_k)\in\R^k$, is a parallel vector 
bundle isometry. 

For a fixed $y\in \mathbb{R}^k$, we denote by
$\Psi(y)$ the parallel section of $\mathcal{L}$ given by $\Psi(y)(x_1) = \Psi(x_1,y)$
for all $x_1\in M_1$. Given an isometric immersion $f_0\colon M_0 \to \R^k$, first recall that the partial tube over $f_1$ with $f_0$ as fiber is the map  $g\colon M_0 \times M_1 \to \R^m$ given by
$$
g(x_0, x_1)=f_1(x_1)+\Psi_{x_1}(f_0(x_0)).
$$
Geometrically, $g(M_0\times M_1)$ is generated by taking the image of $f_0(M_0)$ under a fixed $\Psi_{x_1}$, $x_1\in M_1$, and parallel translating it along $f_1$ with respect to its normal connection (see, e.g., Chapter $10$ of \cite{dt0} for details).  

Now take $(\varphi,\beta) \in \mathcal{D}(f_1)$ and define $f\colon M_0 \times M_1 \to \R^m$ by
\begin{equation}\label{e:ribpartialtube1}
f(x_0,x_1) = \big(\mathcal{R}_{(\varphi,\beta + \Psi(f_0(x_0)))}f_1\big)(x_1)=(f_1 - 2\nu_{x_0}\varphi \mathcal{F}_{x_0})(x_1),
\end{equation}
where $\mathcal{F}_{x_0} = f_1{}_*\text{grad}\,\varphi + \beta + \Psi(f_0(x_0))$ and $\nu_{x_0} 
= ||\mathcal{F}_{x_0}||^{-2}$.  
For each $x_0\in M_0$, we denote by 
$$
\mathcal{P}_{x_0} = I - 2\nu_{x_0}\mathcal{F}_{x_0}^*\mathcal{F}_{x_0}\quad \text{and} \quad D_{x_0} 
= I - 2\nu_{x_0}\varphi\Phi_{x_0},
$$
with $\Phi_{x_0} = \text{Hess}\,\varphi - A_{\beta + \Psi(f_0(x_0))}^{f_1}$, the vector bundle 
isometry and the symmetric endomorphism associated with the Ribaucour transform 
$\mathcal{R}_{(\varphi,\beta + \Psi(f_0(x_0)))}f_1$ of $f_1$. 

If, in particular, $(\varphi,\beta) \in \mathcal{D}(f_1)$ is given by $2\varphi=-1$ 
and $\beta=0$,  then
$$
\mathcal{R}_{(\varphi,\beta + \Psi(f_0(x_0)))}f_1=f_1-2\nu_{x_0}\varphi\mathcal{F}_{x_0}\\
=f_1+ \Psi\left(\frac{f_0(x_0)}{\|f_0(x_0)\|^2}\right),
$$
thus $\mathcal{R}_{(\varphi,\beta + \Psi(f_0(x_0)))}f_1$ reduces to the  partial tube 
over $f_1$ whose fiber is the  composition of $f_0$ with an inversion with respect a hypersphere of unit radius centered at the origin.\vspace{1ex}

Given a product manifold $M = \prod_{i=0}^r M_i$ and $x=(x_0, \ldots, x_r)\in M$, we 
denote by $\tau^x_i \colon M_i \to M$ the inclusion given by
 $\tau_i^x(y_i) = (x_0, \ldots, x_{i-1},y_i, x_{i+1}, \ldots, x_r)$. 

\begin{proposition}\label{p:ribdiff}
The differential of the map $f$ in (\ref{e:ribpartialtube1}) at $x = (x_0,x_1)$ is given by
\begin{equation}\label{e:difm0}
f_* \tau_0^x{}_*X_0 = - 2\nu_{x_0}\varphi\mathcal{P}_{x_0}(\Psi_{x_1}(f_0{}_*X_0))
\end{equation}
and
\begin{equation}\label{e:difm1}
f_* \tau_1^x{}_*X_1 = \mathcal{P}_{x_0}f_1{}_*D_{x_0}X_1,  	
\end{equation}
for all $X_i \in T_{x_i}M_i$, $0 \leq i \leq 1$.
\end{proposition}
\begin{proof}
 Differentiating \eqref{e:ribpartialtube1} we obtain
\begin{align*}
f_* \tau_0^x{}_*X_0 &= 4v_{x_0}^2\varphi\left<\Psi_{x_1}(f_0{}_*X_0),
\mathcal{F}_{x_0}\right> \mathcal{F}_{x_0} - 2\nu_{x_0}\varphi\Psi_{x_1}(f_0{}_*X_0)\\
&= - 2\nu_{x_0}\varphi\mathcal{P}_{x_0}(\Psi_{x_1}(f_0{}_*X_0)).
\end{align*}
Equation \eqref{e:difm1} is part of the assertions in Theorem \ref{t:ribcharac}. 
\end{proof}

  If the map $f$ given by (\ref{e:ribpartialtube1}) is an immersion at any point, then it is called the 
\emph{Ribaucour partial tube} over $f_1$ with fiber $f_0$ associated with $\Psi$ and 
$(\varphi,\beta) \in \mathcal{D}(f_1)$, or the Ribaucour partial tube determined by 
$(f_0, f_1, \Psi, \varphi, \beta)$. It follows from Proposition \ref{p:ribdiff} that 
$f$ is an immersion at $(x_0, x_1)\in M$ if and only if neither 
$\varphi$ nor $\mathcal{F}_{x_0}$ vanish at $x_1$ and the endomorphism 
$D_{x_0}$ of $T_{x_1}M_1$ is invertible. 
We always assume that $f_0\colon M_0 \to \R^k$ is a substantial immersion, for if this is not
the case, say, $f_0(M_0) \subset v + \mathbb{R}^l$ for some $v\in \mathbb{R}^k$ and $l<k$,  
we may replace $(\varphi, \beta)$ by $(\varphi, \beta + \Psi_{x_1}(v))$ and restrict $\Psi$ 
to $M_1 \times \mathbb{R}^l$. 
\begin{proposition}\label{prop:basic}
The following assertions on the Ribaucour partial tube $f$ determined by $(f_0,f_1,\Psi,\varphi,\beta)$ hold:

\begin{enumerate}[(i)]
\item The induced metric is given by
$$
\left<\tau_0^x{}_*X_0 + \tau_1^x{}_*X_1, \tau_0^x{}_*Y_0 + \tau_1^x{}_*Y_1\right>_f 
= 4\nu_{x_0}^2\varphi^2\left<X_0,Y_0\right>_{f_0} + \left<D_{x_0}^2X_1,Y_1\right>_{f_1}.
$$
\item The normal space of $f$ at $(x_0, x_1)$ is
$$
N_fM(x_0,x_1) = \mathcal{P}_{x_0}\big(\mathcal{L}^\perp(x_1) \oplus \Psi_{x_1}(N_{f_0}M_0(x_0))\big).
$$
\item Given $\delta \in \Gamma(\mathcal{L}^\perp)$ and $\zeta \in \Gamma(N_{f_0}M_0)$, 
the shape operators of $f$ with respect to $\hat{\delta}, \hat{\zeta} \in N_f(M)$, defined by  
$$
\hat{\delta}(x_0,x_1) = \mathcal{P}_{x_0}\delta(x_1)\;\;\mbox{and}\;\;\hat{\zeta}(x_0,x_1) 
= \mathcal{P}_{x_0}(\Psi_{x_1}(\zeta(x_0))),
$$
are
\begin{equation}\label{e:shaped1}
A_{\hat{\delta}}^f \tau_0^x{}_* = -\varphi^{-1}\left<\delta,\beta\right>\tau_0^x{}_*, 
\end{equation}
\begin{equation}\label{e:shaped2}
A_{\hat{\delta}}^f \tau_1^x{}_* = \tau_1^x{}_*D_{x_0}^{-1}(A_\delta^{f_1} 
+ 2\nu_{x_0}\left<\beta,\delta\right>\Phi_{x_0}),	
\end{equation}
\begin{equation}\label{e:shapexi1}
A_{\hat{\zeta}}^f \tau_0^x{}_* = - \frac{1}{\varphi}\left( \frac{1}{2\nu_{x_0}}\tau_0^x{}_*A_\zeta^{f_0} + \left<\Psi_{x_1}(\zeta),\mathcal{F}_{x_0}\right>\tau_0^x{}_* \right)
\end{equation}
and
\begin{equation}\label{e:shapexi2}
A_{\hat{\zeta}}^f \tau_1^x{}_* =  \tau_1^x{}_*D_{x_0}^{-1} (A_{\Psi_{x_1}(\zeta)}^{f_1} + 2\nu_{x_0}\left<\Psi_{x_1}(\zeta),\mathcal{F}_{x_0}\right>\Phi_{x_0}).		
\end{equation}
\item The normal connection of $f$ is given by
\begin{equation}\label{e:normalcond1}
^{f}\nabla_{\tau_0^x{}_*X_0}^\perp  \hat{\delta} = 0	,
\end{equation}
\begin{equation}\label{e:normalcond2}
^{f}\nabla_{\tau_1^x{}_*X_1}^\perp  \hat{\delta} = \mathcal{P}_{x_0} {}^{f_1}\nabla_{x_1}^\perp \delta,
\end{equation}
\begin{equation}\label{e:normalconxi}
^{f}\nabla_{\tau_0^x{}_*X_0}^\perp  \hat{\zeta} 
= \mathcal{P}_{x_0}(\Psi_{x_1}({}^{f_0}\nabla_{X_0}^\perp \zeta))	
\end{equation}
and
\begin{equation}\label{e:normalconxi2}
^{f}\nabla_{\tau_1^x{}_*X_1}^\perp  \hat{\zeta} = 0.
\end{equation} 
\end{enumerate}
\end{proposition}

\begin{proof}
Items (i) and (ii) are immediate consequences of  \eqref{e:difm0} and \eqref{e:difm1}.  
For (iii) and (iv), on one hand we have
\begin{equation*}
\bar{\nabla}_{\tau_0^x{}_*X_0} \hat{\delta} = 
-f_*A_{\hat{\delta}}^f \tau_0^x{}_*X_0 + {}^f \nabla_{\tau_0^x{}_*X_0}^\perp \hat{\delta} \quad \text{and} \quad 
\bar{\nabla}_{\tau_1^x{}_*X_1} \hat{\delta} = 
-f_*A_{\hat{\delta}}^f \tau_1^x{}_*X_1 + {}^f \nabla_{\tau_1^x{}_*X_1}^\perp \hat{\delta}.
\end{equation*}
On the other hand, using \eqref{e:difm0} and \eqref{e:difm1} we obtain
\begin{align*}
\bar{\nabla}_{\tau_0^x{}_*X_0} \hat{\delta} 
&= \bar{\nabla}_{\tau_0^x{}_*X_0} (\delta - 2\nu_{x_0}\left<\mathcal{F}_{x_0},\delta\right>\mathcal{F}_{x_0})\\
&=  - 2\nu_{x_0}\left<\beta,\delta\right>\mathcal{P}_{x_0}\Psi_{x_1}(f_0{}_*X_0)\\
&= \varphi^{-1}\left<\beta,\delta\right>f_*\tau_0^x{}_*X_0 
\end{align*}
and
\begin{align*}
\bar{\nabla}_{\tau_1^x{}_*X_1} \hat{\delta} 
&=\bar{\nabla}_{\tau_1^x{}_*X_1} (\delta - 2\nu_{x_0}\left<\mathcal{F}_{x_0},\delta\right>\mathcal{F}_{x_0})\\
&= -\mathcal{P}_{x_0}f_1{}_*(A_\delta^{f_1}X_1 + 2\nu_{x_0}\left<\beta,\delta\right>\Phi_{x_0}X_1) 
+ \mathcal{P}_{x_0}{}^{f_1}\nabla_{X_1}^\perp\delta \\
&= -f_*\tau_1^x{}_*D_{x_0}^{-1}(A_\delta^{f_1}X_1 + 2\nu_{x_0}\left<\beta,\delta\right>\Phi_{x_0}X_1)
+ \mathcal{P}_{x_0}{}^{f_1}\nabla_{X_1}^\perp \delta,  
\end{align*}
which yield \eqref{e:shaped1}, \eqref{e:shaped2}, \eqref{e:normalcond1} and \eqref{e:normalcond2}. 
Repeating the argument for $\hat{\zeta} \in \Gamma(N_fM)$ gives
$$\bar{\nabla}_{\tau_0^x{}_*X_0} \hat{\zeta} 
= -f_*A_{\hat{\zeta}}^f \tau_0^x{}_*X_0 + {}^f\nabla_{\tau_0^x{}_*X_0}^\perp \hat{\zeta}\quad  
\text{and} \quad \bar{\nabla}_{\tau_1^x{}_*X_1} \hat{\zeta} = -f_*A_{\hat{\zeta}}^f \tau_1^x{}_*X_1 
+ {}^f\nabla_{\tau_1^x{}_*X_1}^\perp \hat{\zeta}.$$
Using \eqref{e:difm0}, \eqref{e:difm1} and the fact that $\Psi$ is a parallel 
vector bundle isometry, we obtain
\begin{align*}
\bar{\nabla}_{\tau_0^x{}_*X_0} \hat{\zeta} &= \bar{\nabla}_{\tau_0^x{}_*X_0} (\Psi(\zeta) - 2\nu_{x_0}\left<\Psi(\zeta),\mathcal{F}_{x_0}\right>\mathcal{F}_{x_0})\\
&=-\mathcal{P}_{x_0} \Psi_{x_1}f_0{}_*(A_\zeta^{f_0}X_0 
+ 2\nu_{x_0} \left<\Psi_{x_1}(\zeta),\mathcal{F}_{x_0}\right>X_0) 
+ \mathcal{P}_{x_0}(\Psi_{x_1} ({}^{f_0}\nabla_{X_0}^\perp \zeta))\\
&= (2\nu_{x_0}\varphi)^{-1}f_*\tau_0^x{}_*A_\zeta^{f_0}X_0 
+\varphi^{-1}\left<\Psi_{x_1}(\zeta),\mathcal{F}_{x_0}\right>f_*\tau_0^x{}_*X_0 \\
&\quad+ \mathcal{P}_{x_0}(\Psi_{x_1} ({}^{f_0}\nabla_{X_0}^\perp \zeta)) 
\end{align*}
and 
\begin{align*}
\bar{\nabla}_{\tau_1^x{}_*X_1} \hat{\zeta} &=\bar{\nabla}_{\tau_1^x{}_*X_1} (\Psi(\zeta) - 2\nu_{x_0}\left<\Psi(\zeta),\mathcal{F}_{x_0}\right>\mathcal{F}_{x_0})\\
&= -\mathcal{P}_{x_0}f_1{}_*(A_{\Psi_{x_1}(\zeta)}^{f_1}X_1 
+ 2\nu_{x_0}\left<\Psi(\zeta),\mathcal{F}_{x_0}\right>\Phi_{x_0}X_1)\\
&= -f_* \tau_1^x{}_*D_{x_0}^{-1} (A_{\Psi_{x_1}(\zeta)}^{f_1}X_1 
+ 2\nu_{x_0}\left<\Psi(\zeta),\mathcal{F}_{x_0}\right>\Phi_{x_0}X_1).
\end{align*}
Thus \eqref{e:shapexi1}, \eqref{e:shapexi2}, \eqref{e:normalconxi} 
and \eqref{e:normalconxi2} follow.\qed \vspace{1ex}

 Before we state some  consequences of the preceding proposition, we first recall some terminology. 
 
 A \emph{net} $\mathcal{E} = (E_i)_{i=0, \ldots, r}$ 
on a differentiable manifold $M$ is a decomposition of its tangent bundle 
$TM = \oplus_{i=0}^r E_i$ as a Whitney sum of integrable distributions.  
If $M$ is a Riemannian manifold, the net $\mathcal{E} = (E_i)_{i=0, \ldots, r}$ 
is said to be an \emph{orthogonal net} if the distributions $E_i$ are mutually orthogonal. 
Given an isometric immersion $f\colon M\to \mathbb{R}^m$ of a Riemannian manifold $M$ equipped with a net  $\mathcal{E} = (E_i)_{i=0, \ldots, r}$, then the second fundamental form $\alpha_f$ of $f$ is said to be \emph{adapted to the net} $\mathcal{E}$ if $\alpha_f(X_i,X_j)=0$ whenever $X_i\in \Gamma(E_i)$ and $X_j\in \Gamma(E_j)$ with $1\leq i\neq j\leq r$.

In a product manifold $M = \prod_{i=0}^r M_i$, we have a natural net 
$\mathcal{E} = (E_i)_{i=0, \ldots, r}$, called its \emph{product net}, given by the 
tangent bundles of its factors, that is, $E_i(x) = \tau^x_i{}_* T_{x_i}M_i$
for all $x = (x_0, \ldots, x_r) \in M$. For each $0\leq i\leq r$ we
 denote $$M_{\perp_i} = \prod_{\substack{j=0 \\ j\neq i}}^r M_j$$
and define the projection $\pi_{\perp_i}\colon M \to M_{\perp_i}$ by
$$\pi_{\perp_i}(x_0, \ldots, x_r) = (x_0, \ldots, x_{i-1}, x_{i+1}, \ldots, x_r).$$  
Given $x_i \in M_i$, the map $\mu_{x_i}\colon M_{\perp_i} \to M$ stands for  the inclusion
$$\mu_{x_i}(y_0, \ldots, y_{i-1}, y_{i+1}, \ldots, y_r)
 = (y_0, \ldots, y_{i-1}, x_i, y_{i+1}, \ldots, y_r).$$

  A tangent subbundle $E \subset TM$ is 
$\emph{umbilical}$ if there exists $Z \in \Gamma(E^\perp)$ such that 
\begin{equation}\label{e:defumbilical}
(\nabla_T S)_{E^\perp} = \left<T,S\right>Z
\end{equation}
for all  $T, S \in \Gamma(E)$, where the subscript $E^\perp$ denotes the orthogonal projection onto $E^\perp$. Umbilical distributions are integrable and the leaves are 
totally umbilical submanifolds with mean curvature vector field $Z$.  
If, in addition, $Z$ satisfies
\begin{equation}\label{e:defspherical}
(\nabla_T Z)_{E^\perp} = 0
\end{equation}
for any $T \in \Gamma(E)$, then $E$ is said to be \emph{spherical} and its  
leaves are extrinsic spheres.  

\begin{corollary}\label{cor:rpt}
Let $f\colon M_0\times M_1\to \R^m$ be the Ribaucour partial tube determined by 
$(f_0,f_1,\Psi,\varphi,\beta)$ and let $\mathcal{E} = (E_0, E_1)$ be the product 
net of $M$. Then the following assertions hold:
\begin{itemize}
\item[(i)] The net $\mathcal{E}$ is orthogonal with respect to the metric induced by $f$.
\item[(ii)] The second fundamental form of $f$ is adapted to $\mathcal{E}$.
\item[(iii)] $E_0$ is a spherical distribution.
\end{itemize}
\end{corollary}
\proof The assertions in items $(i)$ and $(ii)$ are immediate consequences of 
parts $(i)$ and $(iii)$ of Proposition \ref{prop:basic}, respectively. 
To prove item $(iii)$, we must show that there exists $Z \in \Gamma(E_1)$ such that
\begin{equation}\label{sph1}
\left<\nabla_S T,X\right> = \left<S,T\right>\left<X,Z\right>
\end{equation}
and
\begin{equation}\label{sph2}
\left<\nabla_S Z,X\right> = 0
\end{equation}
for all $S$, $T \in \Gamma(E_0)$ and $X \in \Gamma(E_1)$.  
We have
\begin{align*}
\bar{\nabla}_{\tau_0^x{}_*X_0}&f_*\tau_0^x{}_*Y_0 \\
&= 	-\bar{\nabla}_{\tau_0^x{}_*X_0}2\nu_{x_0}\varphi\mathcal{P}_{x_0}(\Psi(f_0{}_*Y_0))\\
&=\bar{\nabla}_{\tau_0^x{}_*X_0} (-2\nu_{x_0}\varphi\Psi(f_0{}_*Y_0) 
+ 4\nu_{x_0}^2 \varphi \left<\Psi(f_0{}_*Y_0),\mathcal{F}_{x_0}\right>\mathcal{F}_{x_0})\\
&= 4\nu_{x_0}^2\varphi \left<\Psi_{x_1}(f_0{}_*X_0),
\mathcal{F}_{x_0}\right>\mathcal{P}_{x_0}(\Psi_{x_1}(f_0{}_*Y_0))\\
&\quad + 4\nu_{x_0}^2\varphi \left<\Psi_{x_1}(f_0{}_*Y_0),
\mathcal{F}_{x_0}\right>\mathcal{P}_{x_0}(\Psi_{x_1}(f_0{}_*X_0)) \\
&\quad- 2\nu_{x_0}\varphi \mathcal{P}_{x_0}(\Psi_{x_1}(f_0{}_*\nabla_{X_0}Y_0 + \alpha^{f_0}(X_0,Y_0))\\
&\quad - \varphi^{-1}\left<\tau_0^x{}_*X_0,\tau_0^x{}_*Y_0\right>_f\mathcal{P}_{x_0}(f_1{}_*\text{grad}_1\, \varphi 
+ \beta + \Psi_{x_1}(f_0(x_0)),
\end{align*}
where $\text{grad}_1$ stands for the gradient with respect to the metric induced by $f_1$. Using \eqref{e:difm0} and \eqref{e:difm1} we obtain 
$$\left<\nabla_{\tau_0^x{}_*X_0}\tau_0^x{}_*Y_0,\tau_1^x{}_*Y_1\right> = - \varphi^{-1}\left<\tau_0^x{}_*X_0,\tau_0^x{}_*Y_0\right>_f\left<\tau_1^x{}_*D^{-1}_{x_0}\text{grad}_1\, \varphi,\tau_1^x{}_*Y_1\right>_f.$$
Hence \eqref{sph1} holds with
$
Z(x_0,x_1) = -\tau_1^x{}_*D_{x_0}^{-1}\text{grad}_1 \log \varphi(x_1).
$
Now,  \eqref{sph2} follows from
\begin{align*}
\bar{\nabla}_{\tau_0^x{}_*X_0}f_*Z &= -\bar{\nabla}_{\tau_0^x{}_*X_0}\mathcal{P}_{x_0}f_1{}_*\text{grad}_1 \log \varphi \\
&= -4\nu_{x_0}^2\left<\Psi_{x_1}(f_0{}_*X_0),
\mathcal{F}_{x_0}\right> \left<f_1{}_*\text{grad}_1 \log \varphi,\mathcal{F}_{x_0}\right>\mathcal{F}_{x_0}\\
&\quad  + 2\nu_{x_0}\left<f_1{}_*\text{grad}_1 \log \varphi,\mathcal{F}_{x_0}\right>\Psi_{x_1}(f_0{}_*X_0)\\
&= -\varphi^{-1}\left<f_1{}_*\text{grad}_1 \log \varphi,\mathcal{F}_{x_0}\right>f_*\tau_0^x{}_*X_0\\
&=-\left<f_*Z,f_*Z\right>f_*\tau_0^x{}_*X_0.
\end{align*}\end{proof}

\section{A characterization of Ribaucour partial tubes}
 
The aim of this section is to prove that conditions $(i)$ to $(iii)$ in Corollary \eqref{cor:rpt} characterize Ribaucour partial tubes among immersions $f\colon M_0\times M_1\to \R^m$ of product manifolds. We make use of the following lemma
(see Proposition $9$ of~\cite{t3}).  

\begin{lemma} \label{prop:ext1} Let $f\colon\,M^n\to\R^{m}$ and 
$g\colon\,L^k\to M^{n}$ be isometric immersions. Then the following assertions are equivalent:
\begin{itemize}
\item[(i)]  $g$ is an extrinsic sphere whose mean curvature vector has
length $1/r$, $r>0$,  and
$
\alpha_f(g_*X,Z)=0 
$
for all $X\in \mathfrak{X}(L)$ and $Z\in \Gamma(N_gL)$. \vspace{1ex}
\item[(ii)] There exists $\zeta\in \Gamma(N_gL)$ of length $1/r$, $r>0$, such that the subbundle of $f_*N_gL$ orthogonal to $f_*\zeta$ is constant in $\R^m$ and the map $f\circ g+r^2f_*\zeta$ is constant on $L^k$.\vspace{1ex}
\item[(iii)] $f(g(L))$ is contained in a sphere of radius $r$  and dimension $(m\!-n\!+\!k)$ 
in $\R^m$ whose normal space along $f(g(L))$ is $f_*N_gL$.
\end{itemize}
\end{lemma}

The characterization of Ribaucour partial tubes is as follows.

\begin{theorem}\label{rpt} Let $f\colon M = M_0\times M_1 \to \R^m$ be an immersion   satisfying  conditions $(i)$ to $(iii)$ in Corollary \eqref{cor:rpt}. 
Then, for any fixed $\bar{x}_0\in M_0$, the map $f_1\colon M_1\to \R^m$ given by $f_1=f\circ \mu_{\bar{x}_0}$ is an immersion whose normal bundle $N_{f_1}M_1$ carries a parallel flat 
vector subbundle $\mathcal{L}$, and there exist a parallel vector bundle isometry 
$\Psi\colon M_1 \times \R^k \to \mathcal{L}$, an  immersion $f_0\colon M_0\to \R^k$ 
and $(\varphi,\beta) \in \mathcal{D}(f_1)$ such that $f$ is 
the Ribaucour partial tube determined by $(f_0, f_1, \Psi, \varphi, \beta)$.
\end{theorem}
\begin{proof}   
The normal space of $f_1$ at $x_1\in M_1$ splits orthogonally as 
$$N_{f_1}M_1(x_1) = f_*E_0(\bar{x}_0,x_1) \oplus N_fM(\bar{x}_0,x_1).$$
Let $Z\in \Gamma(E_1)$ be the mean curvature vector field of $E_0$
and let $\mu\colon M_1 \to \R^m$ be defined by $\mu(x_1) = f_*Z(\bar{x}_0, x_1)=f_1{}_*Z_1(x_1)$, where $\mu_{\bar{x}_0}{}_*Z_1=Z\circ \mu_{\bar{x}_0}$. 
By Lemma~\ref{prop:ext1}, for each $x_1\in M_1$ the image by $f$ of the leaf 
$\sigma(x_1)=M_0\times \{x_1\}=\mu_{x_1}(M_0)$ of $E_0$ 
is contained in an $(m-m_1)$-dimensional sphere $S(x_1)$ through $f_1(x_1)$ having 
$f_*E_1$ as its normal bundle along $\sigma(x_1)$, and whose center is the constant value
$$f_1(x_1) + \frac{\mu(x_1)}{||\mu(x_1)||^2}$$ 
along  $\sigma(x_1)$ of the map ${\displaystyle f + \frac{f_*Z}{||Z||^2}}$.
Thus we can parametrize $S(x_1)$ by 
$$\mu(x_1) + t\in (\mu(x_1)+N_{f_1}M_1(x_1))\mapsto f_1(x_1) 
+ \frac{2}{||\mu(x_1) + t||^2}(\mu(x_1) + t),$$
which can be thought of as the restriction to the affine subspace 
$\mu(x_1) \oplus N_{f_1}M_1(x_1)$ of the composition of an inversion with respect to 
a sphere of radius $\sqrt{2}$ centered at the origin and a translation by $f_1(x_1)$. 
Notice that the image of this parametrization misses the point 
$f_1(x_1)=f(\bar{x}_0,x_1) \in S(x_1)$ itself, which is achieved by letting $\|t\|$ 
go to infinity.

For each $x_1\in M_1$, since $f(\sigma(x_1)) \subset S(x_1)$ there exists an immersion 
$h^{x_1}\colon M_0 \to N_{f_1}M_1(x_1)$ such that 
$$
f(x_0,x_1) = f_1(x_1) + \frac{2}{||\mu(x_1) + h^{x_1}(x_0)||^2}(\mu(x_1) + h^{x_1}(x_0)).
$$
Denoting 
$$
\rho(x_0, x_1)=\mu(x_1) + h^{x_1}(x_0),
$$
we may write
\begin{equation}\label{eq:f1}
f = f_1\circ \pi_1 + \frac{2}{||\rho||^2}\rho,
\end{equation}
where $\pi_1\colon M\to M_1$ is the projection.

Given $\theta \in  N_{f_1}M_1(x_1)=f_*E_0(\bar{x}_0,x_1)\oplus N_fM(\bar{x}_0,x_1)$, 
for each $x_0\in M_0$ let
$$
\bar{\theta} = \theta - \frac{2\left<\theta,\rho(x_0, x_1)\right>}{||\rho(x_0, x_1)||^2}\rho(x_0, x_1)	
$$
be the reflection of $\theta$ with respect to the hyperplane orthogonal to the vector $\rho(x_0, x_1)$.
We claim that $\bar{\theta}\in T_{f(x_0,x_1)}S(x_1)=(f_*E_1(x_0, x_1))^\perp$, that is, 
$\langle \bar{\theta}, \gamma\rangle =0$ for all $\gamma\in f_*E_1(x_0,x_1)$.  To prove this, 
consider the decomposition
$$
f_*E_1(x_0,x_1)=(f_*E_1(x_0,x_1)\cap\{f_*Z(x_0,x_1)\}^\perp)\oplus\text{span}\,\{f_*Z(x_0,x_1)\}.
$$
Since the subbundle $f_*E_1\cap\{f_*Z\}^\perp$  is parallel along the leaves of $E_0$ by Lemma \ref{prop:ext1},  
\begin{align*}
f_*E_1(x_0,x_1)\cap\{f_*Z(x_0,x_1)\}^\perp &=f_*E_1(\bar{x}_0,x_1)\cap\{f_*Z(\bar{x}_0,x_1)\}^\perp\\
&=f_1{}_*T_{x_1}M_1\cap \{\mu(x_1)\}^\perp.
\end{align*}
Thus $\langle \bar{\theta},\gamma\rangle =0$
for all $\gamma\in f_*E_1(x_0,x_1)\cap\{f_*Z(x_0,x_1)\}^\perp$. On the other hand,
Lemma \ref{prop:ext1} also says that
$$
f(x_0, x_1)+\frac{f_*Z(x_0, x_1)}{\|Z(x_0, x_1)\|^2}=f_1(x_1)+\frac{\mu(x_1)}{\|\mu(x_1)\|^2}
$$
for all $x_0\in M_0$. Moreover,  $\|Z(x_0, x_1)\|$ also does not depend on $x_0$, 
and hence coincides with $\|\mu(x_1)\|$. Hence, substituting \eqref{eq:f1} 
in the preceding equation gives
$$
f_*Z(x_0, x_1)=\mu(x_1)-\frac{2\|\mu(x_1)\|^2}{\|\rho(x_0, x_1)\|^2}\rho(x_0, x_1).
$$ 
It follows that $\langle \bar{\theta},f_*Z(x_0,x_1)\rangle=0,$
and the claim follows.

Now, differentiating \eqref{eq:f1} at $x = (x_0,x_1)$ gives
\begin{align*}
f_*\tau_0^x{}_*X_0 &= \frac{2}{||\rho||^2}\big(h^{x_1}{}_*X_0
-\frac{2\left<h^{x_1}{}_*X_0,\rho\right>}{||\rho||^2}\rho\big)
\end{align*}
and
\begin{align*}
f_*\tau_1^x{}_*X_1 &= f_1{}_*X_1 + \frac{2}{||\rho||^2}\big(\rho_*\tau_1^x{}_*X_1
-\frac{2\left<\rho_*\tau_1^x{}_*X_1,\rho\right>}{||\rho||^2}\rho\big)
\end{align*}

Given $\theta\in N_{f_1}M_1(x_1)$, $x_0\in M_0$ and $X_1\in T_{x_1}M_1$, endowing $M_1$ with the metric induced by $f_1$  we obtain
$$
\left<\rho_*\tau_1^x{}_*X_1, \theta\right>=\left<Z_1, X_1\right>\left<h^{x_1}(x_0), \theta\right>,
$$
bearing in mind that $\left<\bar{\theta}, f_*\tau_1^x{}_*X_1\right>=0$. Thus
\begin{equation}\label{e:paramcondresult}
\alpha^{f_1}(X_1, Z_1)+\nabla_{X_1}^\perp h(x_0)= \left<Z_1, X_1\right>h^{x_1}(x_0)
\end{equation}
for all $x_0\in M_0$ and $X_1\in T_{x_1}M_1$.

   For a fixed $z_0 \in M_0$, define $\xi^{x_0}=h(x_0)-h(z_0)$. Then
$$
\nabla_{X_1}^\perp \xi^{x_0} = \left<X_1, Z_1\right>\xi^{x_0}
$$
for all $X_1 \in T_{x_1}M_1$. In particular, this implies that
$$
\left<X_1, Z_1\right>=X_1(\log \|\xi^{x_0}\|)
$$
for all $X_1 \in T_{x_1}M_1$, that is, $Z_1=\text{grad}\, \tau$, where
$\tau=\log \|\xi^{x_0}\|$. Hence $e^{-\tau}\xi^{x_0}$ is a parallel
section of $N_{f_1}M_1$ for any $x_0\in M_0$. It follows that the subspaces
$$\mathcal{L}(x_1) = \text{span}\,\{\xi^{x_0}(x_1): x_0 \in M_0\} \subset N_{f_1}M_1(x_1)$$
define a parallel and flat subbundle of $N_{f_1}M_1$.  

Let $\Psi\colon M_1 \times \R^k \to \mathcal{L}$ be a parallel vector bundle isometry.  
Since $e^{-\tau}\xi^{x_0}$ is a parallel section of $\mathcal{L}$, 
there exists $f_0\colon M_0 \to \R^k$ such that $\Psi(f_0(x_0)) = e^{-\tau}\xi^{x_0}$.

Let $\varphi \in C^\infty(M_1)$ be given by $\varphi = - e^{-\tau}$ and let $\beta 
= e^{-\tau}h(z_0) \in \Gamma(N_{f_1}M_1)$.  
Then \eqref{e:paramcondresult} becomes
$$
\alpha^{f_1}(X_1, \text{grad}\,\varphi)+\nabla_{X_1}^\perp \beta=0,
$$
that is, $(\varphi, \beta)\in \mathcal{D}(f_1)$.  Finally, 
\begin{align*}
\mathcal{R}_{(\varphi,\beta + \Psi(f_0(x_0)))}f_1 &= f_1 - 
\frac{2(-e^{-\tau})}{e^{-2\tau}||\mu + h(x_0)||^2}(e^{-\tau}(\mu + h(x_0))) \\
&= f_1 + \frac{2}{||\mu + h(x_0)||^2}(\mu + h(x_0)), \\
&= f,
\end{align*}
thus $f$ is the Ribaucour partial tube determined by $(f_0,f_1,\Psi,\varphi,\beta)$.  
\end{proof}

\begin{remarks}\emph{ $1)$ Theorem\ref{rpt} can be regarded as a conformal counterpart of 
Theorem~$3.5$ in \cite{t4}, which  characterizes partial tubes as the immersions $g\colon M_0\times M_1\to \R^m$ of product manifolds whose induced metrics have  properties $(i)$ and $(ii)$ of Corollary \ref{cor:rpt} and, in addition, the distribution $E_0$ in the product net $\mathcal{E} = (E_0, E_1)$ of $M_0\times M_1$ is totally geodesic. A preliminary step in the proof of that result is the version of Lemma \ref{prop:ext1} which states that, if $g\colon\,M^n\to\R^{m}$ and $h\colon\,L^k\to M^{n}$ are isometric immersions, then  $h$ is totally geodesic and $\alpha_g(h_*X,Z)=0$ for all $X\in \mathfrak{X}(L)$ and $Z\in \Gamma(N_hL)$ if and only if   $g(h(L))$ is contained in an affine subspace of $\R^m$ whose normal space along $g(h(L))$ is $g_*N_hL$. In particular, if $g\colon M_0\times M_1\to \R^m$ is a partial tube, then the image $g(M_0\times \{x_1\})$ by $g$ of each leaf of $E_0$ is contained in an affine subspace of $\R^m$ whose normal space along
$g(M_0\times \{x_1\})$ is $g_*E_1$. Similarly, it follows from Corollary \ref{cor:rpt} and Lemma \ref{prop:ext1} that if  $f\colon M_0\times M_1\to \R^m$ is a Ribaucour partial tube, then the image $f(M_0\times \{x_1\})$ by $f$ of each leaf of $E_0$ is contained in a sphere of $\R^m$ whose normal space along $f(M_0\times \{x_1\})$ is $f_*E_1$. As a consequence, the composition of a partial tube $g\colon M_0\times M_1\to \R^m$ with an inversion $I$ with respect to a hypersphere of $\R^m$ is a Ribaucour partial tube, and a Ribaucour partial tube $f\colon M_0\times M_1\to \R^m$ is given in this way  if and only if the spheres containing the images $f(M_0\times \{x_1\})$ by $f$ of the leaves of $E_0$ all pass through a common point.\vspace{1ex}\\
$2)$ The definition of a Ribaucour partial tube, as well as Proposition \ref{prop:basic},
Corollary \ref{cor:rpt} and Theorem \ref{rpt}, can be easily extended for the case in which the ambient space is any space form $\mathbb{Q}_c^m$, by making use of the extension of the Ribaucour transformation to this setting discussed at the end of Section $2$. The details are left to the reader.}
\end{remarks}

\begin{corollary}\label{cor:sphdist}
Let $M$ be a Riemannian manifold carrying a spherical distribution $D$ 
whose orthogonal distribution $D^\perp$ is integrable. 
Then any isometric immersion  $f\colon M \to \R^m$ whose
second fundamental form  is adapted to the net $\mathcal{E} = (D, D^\perp)$ 
is locally a Ribaucour partial tube over the restriction of $f$ to a leaf of $D^\perp$.	
\end{corollary}

In particular, Corollary \ref{cor:rpt} and Theorem \ref{rpt} yield the following 
explicit parametrization of any umbilic-free surface with flat normal bundle of a 
space form such that the lines of curvature of one of the two families have constant 
geodesic curvature.

\begin{corollary}\label{cor:surfaces}
Let $\gamma\colon I\to \mathbb{R}^{n+1}$ be a unit-speed curve. 
Let $\xi_1, \ldots, \xi_{n}$ be a parallel orthonormal frame of $N_{\gamma}I$ 
and let $\varphi, \beta_i\in C^{\infty}(I)$, 
$1\leq i\leq n$, satisfy $\beta_i'+\varphi' k_i=0$, 
where $\gamma''=\sum_{i=1}^{n}k_i\xi_i$.
Let $\alpha\colon J\to \mathbb{R}^{n}$ be a  unit-speed curve. 
Then the map $f\colon I\times J\to \mathbb{R}^{n+1}$ given by
\begin{equation}\label{e:paramsurf}
f(s,t)=\gamma(s)-2\varphi(s)\frac{\varphi'(s)\gamma'(s)
+\sum_{i=1}^{n} (\beta_i(s)+\alpha_i(t))\xi_i(s)}{(\varphi'(s))^2
+\sum_{i=1}^{n} (\beta_i(s)+\alpha_i(t))^2}
\end{equation}
parametrizes, at regular points, a surface with flat normal bundle whose 
coordinate curves are lines of curvature and such that the $t$-coordinate 
curves have constant geodesic curvature.

 Conversely, any umbilic-free surface with flat normal bundle whose 
lines of curvature of one family have constant geodesic curvature can be locally parametrized 
in this way.
\end{corollary}
\begin{proof} That $\xi_{i}$ is parallel along $\gamma$ in the normal connection 
means that there exists $k_i\in C^{\infty}(I)$ such that 
$\xi_{i}'=-k_i\gamma'$. This implies that $\gamma''=\sum_{i=1}^{n}k_i\xi_i$. Therefore, 
for $\varphi\in C^{\infty}(I)$  and $\beta=\sum_{i=1}^{n}\beta_i\xi_i$, Eq. 
\eqref{e:combescuresolution} reduces to the set of ODEs 
$\beta_i'+\varphi' k_i=0$, $1\leq i\leq n$.
Since the map $f\colon I\times J\to \mathbb{R}^{n+1}$ given by
\eqref{e:paramsurf} is the Ribaucour partial tube over $\gamma$ with 
$\alpha\colon J\to \mathbb{R}^{n}$ as fiber, the assertion in the direct
statement is a consequence of Corollary \ref{cor:rpt}, while the converse
follows from Theorem \ref{rpt}.
\end{proof}

   Given an isometric immersion $f\colon M^n\to\tilde{M}^m$, a vector 
$\eta\in N_fM(x)$ at $x\in M^n$ is called a  
\emph{principal  normal} of $f$ at $x$ if the subspace 
$$
E_\eta(x)=\{T\in T_xM:\alpha(T,X)=\left<T,X\right>\eta\;\;\mbox{for all}\;\;X\in T_xM\}
$$
is nontrivial. A normal vector field $\eta\in \Gamma(N_fM)$ is called a
\emph{principal  normal vector field of $f$ with multiplicity $q>0$}
if $E_\eta(x)$ has  dimension $q$ at any point $x\in M^n$.
A principal normal vector field  $\eta\in\Gamma(N_fM)$
is said to be a \emph{Dupin principal normal vector field} if $\eta$ is parallel 
in the normal connection along $E_\eta$.\vspace{1ex}

  As a particular case of Corollary \ref{cor:sphdist}, we recover one of the main results
in \cite{dft}. 

\begin{corollary}\label{dupin} Let $f\colon M^n\to\mathbb{R}^m$ be an isometric immersion 
that carries a Dupin principal normal vector field $\eta$ with multiplicity $q>0$. 
If the subbundle $E_\eta^\perp$ is integrable, then the restriction $\mathcal{N}=E_\eta|_{M_1}$ of $E_\eta$ to any leaf $M_1^{n-q}$ of $E_\eta^\perp$ is a flat parallel  
subbundle of $N_{f_1}M_1$, where $f_1=f|_{M_1}$, and there exist $(\varphi, \beta)\in \mathcal{D}(f_1)$ and a vector bundle isometry $\Psi\colon M_1\times \R^{q}\to \mathcal{N}$ such that $f$ is locally a Ribaucour partial tube determined by $(f_0,f_1,\Psi,\varphi,\beta)$, where $f_0\colon \R^{q}\to \R^{q}$ is the identity map.
\end{corollary}

  Corollary \ref{dupin} yields as a special case the following description of all \emph{channel hypersurfaces}, that is, hypersurfaces $f\colon M^n\to\mathbb{R}^{n+1}$, $n\geq 2$, carrying a principal curvature $\lambda$ with  multiplicity $n-1$, which is constant along the correspondent lines of curvature if $n=2$ ($\lambda$ is automatically constant along the leaves of its eigendistribution if $n\geq 3$). If $n\geq 4$, these are precisely the  conformally flat hypersurfaces of $\mathbb{R}^{n+1}$ (see, e.g., Theorem $9.6$ in \cite{dt0} for an alternative description based on the conformal Gauss parametrization).
 
\begin{corollary}\label{cor:channel}
Let $f\colon M^n\to\mathbb{R}^{n+1}$, $n\geq 2$, be a hypersurface carrying a principal curvature $\lambda$ with multiplicity $n-1$, which is constant along the correspondent lines of curvature if $n=2$. Then the restriction $\mathcal{N}=E_\lambda|_{M_1}$ of the eigendistribution $E_\lambda$ to any integral curve $M_1$ of $E_\lambda^\perp$ 
 is a flat parallel subbundle of $N_{f_1}M_1$, where $f_1=f|_{M_1}$, and there exist $(\varphi, \beta)\in \mathcal{D}(f_1)$ and a vector bundle isometry $\Psi\colon M_1\times \R^{n-1}\to \mathcal{N}$ such that $f$ is locally a Ribaucour partial tube determined by $(f_0,f_1,\Psi,\varphi,\beta)$, where $f_0\colon \R^{n-1}\to \R^{n-1}$ is the identity map.
\end{corollary}

 Ribaucour partial tubes also provide a parametrization of all foliations of a space form
by spherical submanifolds whose orthogonal distributions are integrable, in particular
of all foliations of a space form whose leaves are spherical hypersurfaces. 

\begin{corollary}\label{rptdiffeo} Let $f\colon M^m \to \mathbb{Q}_c^m$ be a local 
diffeomorphism of a product manifold $M^m = M^{m_0}_0\times M^{m_1}_1$ such that the product net 
$\mathcal{E} = (E_0, E_1)$ of $M^m$ is orthogonal and  $E_0$ is  spherical 
 with respect to the metric induced by $f$. Then, for any fixed $\bar{x}_0\in M_0$, the map $f_1\colon M_1\to \R^m$ given by $f_1=f\circ \mu_{\bar{x}_0}$ is an immersion  with flat normal bundle and there exist a parallel vector bundle isometry 
$\Psi\colon M^{m_1}_1 \times \R^{m_0} \to N_{f_1}M_1$, a local isometry 
$f_0\colon M_0\to \R^{m_0}$ and $(\varphi,\beta) \in \mathcal{D}(f_1)$ such that $f$ is 
the Ribaucour partial tube determined by $(f_0, f_1, \Psi, \varphi, \beta)$.
\end{corollary}

\section{Hypersurfaces of space forms carrying a spherical foliation}

   In this section we derive some consequences of Theorem \ref{rpt} for 
hypersurfaces of space forms, in particular we provide a complete solution of 
the problem that was the initial motivation of this work.   

 A smooth  distribution $D$ on a Riemannian manifold 
$M^n$ is said to be \emph{curvature invariant} if 
$R(X,Y)Z \in \Gamma(D)$ for all $X,Y,Z\in \Gamma(D)$,
where $R$ denotes the curvature tensor of $M^n$.  A basic observation for us is the 
following fact.

\begin{lemma} \label{le:sphinv}
Any spherical distribution on a Riemannian manifold is curvature invariant.
\end{lemma}
\begin{proof} Let $D$ be a spherical distribution on a Riemannian manifold $M^n$
with mean curvature vector field $Z$. 
Given $T, S, U\in \Gamma(D)$ and 
$X\in \Gamma(D^\perp)$ with $\left<X, Z\right>=0$, since
$\left<\nabla_S U, X\right>=0$ and $\nabla_T X=0$ by \eqref{e:defumbilical}
and \eqref{e:defspherical}, respectively, it follows that $\left<R(T,S)U,X\right>=0$. 
On the other hand, again from  \eqref{e:defumbilical}
and \eqref{e:defspherical} we obtain
\begin{align*} \left<\nabla_T\nabla_S U, Z\right>&
=T(\left<S,U\right>)\|Z\|^2-\left<\nabla_S U,\nabla_T Z\right>\\
&=(T(\left<S,U\right>)+\left<\nabla_S U,T \right>)\|Z\|^2.
\end{align*}
Similarly,
$$
\left<\nabla_S\nabla_T U, Z\right>=S(\left<T,U\right>)+\left<\nabla_T U,S \right>)\|Z\|^2
$$
and
$$
\left<\nabla_{[T,S]} U, Z\right>=\left<[T,S],U\right>\|Z\|^2.
$$
Subtracting the last two equations from the first one gives $\left<R(T,S)U,Z\right>=0$.
\end{proof}
 
For an oriented  hypersurface $f\colon M^{n}\to\mathbb{Q}_c^{n+1}$, 
we  denote by $A$ its shape operator with respect to the Gauss map $N$ and by 
$\Delta=\ker A$ its relative nullity distribution. 
The next algebraic lemma was proved in \cite{drt}.

\begin{lemma}\label{prop:curv}
Let $f\colon M^{n}\to\mathbb{Q}_c^{n+1}$ be an oriented hypersurface carrying 
a curvature invariant distribution  $D$ of rank $k$.
Then one of the following holds pointwise:
\begin{itemize}
\item[(i)] $A(D)\subset D^\perp$,
\item[(ii)] $A(D)\subset D$,
\item[(iii)] $\text{rank}\, D\cap \Delta=k-1$.
\end{itemize}
\end{lemma} 

   The two preceding lemmas have the following consequence for hypersurfaces of 
dimension $n\geq 3$ that carry a spherical foliation of high rank.

\begin{corollary} \label{cor:sphinv}
If $f\colon M^{n}\to\mathbb{Q}_c^{n+1}$, $n\geq 3$, is an oriented hypersurface 
and $D$ is a spherical distribution on $M^n$ of rank $k>n/2$ that is not totally 
geodesic on any open subset, then $D$ is invariant by the shape operator of $f$. 
\end{corollary}
\begin{proof} Since the subset of $M^n$ where $D$ is invariant by the shape 
operator of $f$ is closed, there is no loss of generality in assuming that its
mean curvature vector field never vanishes. Then $D\cap \Delta$ must be trivial, 
for $\Delta$ is totally geodesic. On the other hand, the distribution $D$ being 
curvature invariant by Lemma \ref{le:sphinv}, it follows from  Lemma \ref{le:sphinv} 
that one of conditions $(i)$ to $(iii)$ must hold at any point of $M^n$.
Since $k>n/2$, then $\text{rank}\, D(x)\cap \Delta(x) >0$ if $A(D(x))\subset D(x)^\perp$
at some $x\in M^n$, and because $k-1>0$, then $\text{rank}\, D(x)\cap \Delta(x) >0$ 
also if $(iii)$ holds at $x$. Thus neither condition $(i)$ nor condition $(iii)$ can 
hold at any point of $M^n$. 
\end{proof}

    It follows from Corollaries \ref{cor:sphdist} and \ref{cor:sphinv} that any 
hypersurface $f\colon M^{n}\to\mathbb{Q}_c^{n+1}$, $n\geq 3$, that carries a spherical 
distribution $D$ of rank $k>n/2$ whose mean curvature vector field never vanishes and
such that $D^\perp$ is integrable is locally a Ribaucour partial tube over a leaf of $D^\perp$.
In particular, this leads to the following  explicit description of such hypersurfaces
when $k=n-1$ (in which case integrability of $D^\perp$ is automatic).    

\begin{corollary}\label{cor:hyp}
Let $\gamma\colon I\to \mathbb{R}^{n+1}$ be a unit-speed curve. 
Let $\xi_1, \ldots, \xi_{n}$ be a parallel orthonormal frame
of $N_{\gamma}I$ and let $\varphi, \beta_i\in C^{\infty}(I)$, 
$1\leq i\leq n$, satisfy $\beta_i'+\varphi' k_i=0$, 
where $\gamma''=\sum_{i=1}^{n-1}k_i\xi_i$. 
Let $g\colon M_1^{n-1}\to \mathbb{R}^{n}$ be a hypersurface. 
Then the map  $f\colon M^n=I\times M_1^{n-1}\to \mathbb{R}^{n+1}$ 
given by
$$
f(s,t)=\gamma(s)-2\varphi(s)\frac{\varphi'(s)\gamma'(s)
+\sum_{i=1}^{n} (\beta_i(s)+g_i(t))\xi_i(s)}{(\varphi'(s))^2
+\sum_{i=1}^{n} (\beta_i(s)+g_i(t))^2}
$$
parametrizes, at regular points, a hypersurface for which the 
tangent spaces to $M_1$ give rise to a spherical distribution 
of codimension one.

 Conversely, any hypersurface  of $\R^{n+1}$, $n\geq 3$,  
carrying a spherical distribution of codimension one can be 
locally parametrized in this way.
\end{corollary}

\section{A de Rham-type theorem for product manifolds}

  In the last section of this article we will derive from Theorem \ref{rpt} a  decomposition theorem for conformal immersions of product manifolds endowed with polar metrics (see the end of the introduction for the definition of a polar metric). As a preliminary step of independent interest, in this section we prove a de Rham-type theorem of an intrinsic nature which characterizes the  Riemannian manifolds that are conformal to product manifolds endowed with polar metrics. 

 It was shown in Proposition 4 in \cite{t1} that a Riemannian metric on a product manifold 
$M = \prod_{i=0}^r M_i$ is polar if and only if the product net 
$\mathcal{E} = (E_i)_{i=0, \ldots, r}$ of $M$ is an orthogonal net such that $E_a^\perp$ 
is totally geodesic for all $1 \leq a \leq r$. Our first goal is to obtain a similar 
characterization of metrics on a product manifold that are conformal to a polar metric. 
Recall that two Riemannian metrics $g_1$ and $g_2$ on a manifold $M$ are \emph{conformal} 
if there exists a positive $\lambda \in  C^\infty(M)$ such that $g_2 = \lambda^2 g_1$. 
The function $\lambda$ is called the \emph{conformal factor} of $g_2$ with respect to~$g_1$.  

\begin{theorem}\label{p:polarmetricumblicialnet}
A Riemannian metric on a product $M = \prod_{i=0}^r M_i$ of connected manifolds is conformal 
to a polar metric if and only if the product net $\mathcal{E} = (E_i)_{i=0, \ldots, r}$ 
of $M$ is an orthogonal net such that $E_a^\perp$ is totally umbilical for $1 \leq a \leq r$.
\end{theorem}
\begin{proof} Let $g_1$ be a polar metric on $M = \prod_{i=0}^r M_i$ and let $g_2$ be conformal to $g_1$ with conformal factor $\lambda\in C^{\infty}$. 
It is well known that the Levi-Civita conections 
$\nabla^1$ and $\nabla^2$ of $g_1$ and $g_2$, respectively, are related by
\begin{equation}\label{e:levicivitarelationconf}
\nabla_X^2 Y = \nabla_X^1 Y + \frac{1}{\lambda}\left(Y(\lambda)X + X(\lambda)Y 
- g_2(X,Y)\text{grad}_2\lambda \right),
\end{equation} 
where $\text{grad}_2$ denotes the gradient with respect to $g_2$. Since $E_a^\perp$ is a totally geodesic distribution with respect to $g_1$ for $1 \leq a \leq r$, it follows from
\eqref{e:levicivitarelationconf} that $E_a^\perp$ is  totally umbilical  
with respect to $g_2$ with mean curvature vector field $-(\text{grad}_2 \log \lambda)_{E_a}$.  

To prove the converse statement, 
for each $1\leq a\leq r$ decompose $M=M_a\times M_{\perp_a}$. 	
Fix $p=(p_0,\ldots,p_r)\in M$ 
and endow $M_{\perp_a}$ with the metric $g_{\perp_a}=\mu_{p_a}^*g$.  
Given any $x=(x_0,\ldots,x_r)$, denote 
$$
p^a=(x_0,\ldots,x_{a-1},p_a,x_{r+1},\ldots,x_r),\;\; 
x^a=(p_0,\ldots,p_{a-1},x_a,p_{a+1},\ldots,p_r)
$$
and
$$
x^{0,a}=(x_0,p_1,\ldots,p_{a-1},x_a,p_{a+1},\ldots,p_r).
$$
Since $E_a^\perp$ is totally umbilical, it follows from Proposition 1 
in \cite{pr} that there exists a positive $\lambda_a\in C^{\infty}(M)$ such that
$$
\mu_{x_a}^*g=\lambda^2_{x_a}g_{\perp_a}=\lambda^2_{x_a}\mu_{p_a}^*g,
\;\;\mbox{for}\;\; 1\leq a\leq r,
$$ 
where $\lambda_{x_a}=\lambda_a\circ\mu_{x_a}$.   Thus, for all  $0\leq b\neq a\leq r$ and 
$X_b$, $Y_b\in T_{x_b}M_b$ we have 
\begin{equation}\label{sheet}
g(x_0,\ldots,x_r)(\tau_b^x{}_*X_b,\tau_b^x{}_*Y_b) 
=\lambda^2_{x_a}(\pi_{\perp_a}(x))g(p^a)(\tau_b^{p^a}\!{}_*X_b,\tau_b^{p^a}\!{}_*Y_b).
\end{equation}
For a fixed $1\leq a\leq r$, applying \eqref{sheet} for each  $1\leq b\neq a\leq r$ we obtain
\begin{align}\label{eq:dr2}
g(x_0,\ldots,x_r)&(\tau_a^x{}_*X_a,\tau_a^x{}_*Y_a)\\
&=\lambda^2_{x_r}(x_0,\ldots,x_{r-1})
\cdots\lambda_{x_{a+1}}^2(x_0,\ldots,x_a, p_{a+2}, 
\ldots,p_r)\nonumber\\
&\quad\lambda_{x_{a-1}}^2(x_0,\ldots,x_{a-2},x_a,p_{a+1},\ldots,p_r)\cdots\nonumber\\
&\quad\cdots\lambda_{x_1}^2(x_0,p_2,\ldots,p_{a-1},x_a,p_{a+1}, 
\ldots,p_r)\nonumber\\
&\quad g(x^{0,a})(\tau_a^{x^{0,a}}{}_* X_a,\tau_a^{x^{0,a}}{}_*Y_a)\nonumber.
\end{align}
On the other hand, for $b=0$, using recursively \eqref{sheet} following 
the order given by an arbitrary permutation 
$s\colon\{1,\ldots,r\}\to\{1,\ldots,r\}$,
we obtain
\begin{align}\label{eq:dr1}
g(x_0,&\ldots,x_r)(\tau_0^x{}_*X_0,\tau_0^x{}_*Y_0)\\
&=\lambda^2_{x_{s(1)}}(x_0,\ldots,x_{{s(1)}-1},x_{{s(1)}+1},\ldots,x_r)
\nonumber\cdots\\
&\quad\cdots\lambda^2_{x_{s(r)}}(x_0,p_1,\ldots,p_{s(r)-1},p_{s(r)+1}, 
\ldots,p_r)
g(x^0)(\tau_0^{x^0}{}_*X_0,\tau_0^{x^0}{}_* Y_0)\nonumber.
\end{align}
Since the permutation $s$ in \eqref{eq:dr1} is arbitrary, then for any two such permutations 
 $s, t$  we have
\begin{align}\label{equality}
&\lambda^2_{x_{s(1)}}(x_0,\ldots,x_{{s(1)}-1},x_{{s(1)}+1},\ldots,x_r)
\cdots\lambda^2_{x_{s(r)}}(x_0,p_1,\ldots,p_{s(r)-1},p_{s(r)+1}, 
\ldots,p_r)\\
&=\lambda^2_{x_{t(1)}}(x_0,\ldots,x_{{t(1)}-1},x_{{t(1)}+1},\ldots,x_r)
\cdots\lambda^2_{x_{t(r)}}(x_0,p_1,\ldots,p_{t(r)-1},p_{t(r)+1},\ldots,p_r)\nonumber
\end{align}

 Consider the metric $g_0=\tau^{p}_0{}^*g$ on $M_0$ and, for each $1\leq a\leq r$ and $x_0\in M_0$,  let $g_a(x_0)$ be the metric on $M_a$ given by
$$
g_a(x_0)(x_a)=\lambda^{-2}_{x_a}(x_0,p_1,\ldots,\hat{p}_a,\ldots,p_r)
(\tau_a^{x^0}{}^*g)(x_a).
$$
 The proof will be completed as soon as we show that
 \begin{equation}\label{conformalp}
g=\varphi^2(\pi_0^*g_0+\sum_{b=1}^r\pi_b^*(g_b\circ\pi_0)),
\end{equation}
where 
$$
\varphi^2(x_0,\ldots,x_r)=
\lambda^2_{x_r}(x_0,\ldots,x_{r-1})\cdots
\lambda^2_{x_1}(x_0,p_2,\ldots,p_r).
$$
Using \eqref{eq:dr1} for the permutation $s(\kappa)=r+1-\kappa$
we obtain
\begin{align*}
g(x)(&\tau_0^x{}_*X_0,\tau_0^x{}_*Y_0)\\
&=\lambda^2_{x_r}(x_0,\ldots,x_{r-1})\cdots
\lambda^2_{x_1}(x_0,p_2,\ldots,p_r)
g(x^0)(\tau_0^{x^0}{}_*X_0,\tau_0^{x^0}{}_*Y_0),
\end{align*}
while from the definition of the $g_0$ we have
\begin{align*}
\varphi^2(\pi_0^*g_0)(x)(\tau_0^x{}_*X_0,\tau_0^x{}_*Y_0)
&=\varphi^2(x)g_0(x_0)(X_0,Y_0)\\
&=\varphi^2(x)g(x^0)(\tau_0^p{}_*X_0,\tau_0^p{}_*Y_0)\\
&=\varphi^2(x)g(x^0)(\tau_0^{x^0}{}_*X_0,\tau_0^{x^0}{}_*Y_0).
\end{align*}
Thus \eqref{conformalp} holds for pairs of vectors that belong to $E_0$.  
Since the product net is orthogonal, we only have to show 
that \eqref{conformalp} holds for pairs of vectors in  
$E_a$ for any $1\leq a\leq r$.  From the definition of the metric
$g_a(x_0)$ for $x_0\in M_0$ we have
\begin{align*}
\pi_a^*(g_a\circ\pi_0)(x)(\tau_a^x{}_*X_a,&\tau_a^x{}_*Y_a)=
g_a(x_0)(x_a)(X_a,Y_a)\\
&=\lambda^{-2}_{x_a}(x_0,p_1,\ldots,\hat{p}_a,\ldots,p_r)
(\tau_a^{x^0}{}^*g)(x_a)	(X_a,Y_a)\\
&=\lambda^{-2}_{x_a}(x_0,p_1,\ldots,\hat{p}_a,\ldots,p_r)g(x^{0,a})
(\tau_a^{x^0}{}_*X_a,\tau_a^{x^0}{}_*Y_a)\\
&=\lambda^{-2}_{x_a}(x_0,p_1,\ldots,\hat{p}_a,\ldots,p_r)g(x^{0,a})
(\tau_a^{x^{0,a}}{}_*X_a,\tau_a^{x^{0,a}}{}_*Y_a),
\end{align*}
whereas from the definition of  $\varphi$ and \eqref{equality} 
we obtain
\begin{align*}
\varphi^2(x)&=\lambda^2_{x_r}(x_0,\ldots,x_{r-1})
\cdots\lambda_{x_{a+1}}^2(x_0,\ldots,x_a, p_{a+2}, 
\ldots,p_r)\nonumber\\
&\quad\lambda_{x_{a-1}}^2(x_0,\ldots,x_{a-2},x_a,p_{a+1},\ldots,p_r)\cdots\nonumber\\
&\quad\cdots\lambda_{x_1}^2(x_0,p_2,\ldots,p_{a-1},x_a,p_{a+1}, 
\ldots,p_r)\lambda_{x_a}^2(x_0,p_1,\ldots,\hat{p}_a,\ldots,p_r).	
\end{align*}
The last two equalities give
\begin{align*}
\varphi^2(x)(\pi_a^*(g_a&\circ\pi_0)(x)(\tau_a^x{}_*X_a,\tau_a^x{}_*Y_a))\\
&=\lambda^2_{x_r}(x_0,\ldots,x_{r-1})
\cdots\lambda_{x_{a+1}}^2(x_0,\ldots,x_a, p_{a+2}, 
\ldots,p_r)\nonumber\\
&\quad\lambda_{x_{a-1}}^2(x_0,\ldots,x_{a-2},x_a,p_{a+1},\ldots,p_r)\cdots\nonumber\\
&\quad\cdots\lambda_{x_1}^2(x_0,p_2,\ldots,p_{a-1},x_a,p_{a+1}, 
\ldots,p_r)g(x^{0,a})
(\tau_a^{x^{0,a}}{}_*X_a,\tau_a^{x^{0,a}}{}_*Y_a),
\end{align*}
thus proving \eqref{eq:dr2} and completing the proof.
\end{proof}

  The following additional fact will be needed in the proof of Theorem \ref{thm:decomp} below.
  
\begin{proposition}\label{prop:hess}
If $g_1$ is a polar metric on a product manifold $M = \prod_{i=0}^r M_i$ with product net 
$\mathcal{E} = (E_i)_{i=0, \ldots, r}$ and $g_2$ is conformal to $g_1$ with conformal factor 
$\lambda\in C^{\infty}(M)$, then $E_0$ is a spherical distribution if and only if 
$\text{Hess}\,\lambda$ is adapted to the net $(E_0, E_0^\perp)$.
\end{proposition}
\begin{proof}
Since $E_0$ is totally geodesic with respect to $g_1$, by \eqref{e:levicivitarelationconf} it is 
umbilical with with mean curvature vector field $\eta=-(\text{grad}\log \lambda)_{E_0^\perp}$ with respect 
to $g_2$.  Now, for $X_0 \in E_0$ and $Y_j \in E_j \subset E_0^\perp$, $j \neq 0$, we have
\begin{align*}
-\left<\nabla_{X_0}\eta,Y_j\right> 
&= \left<\nabla_{X_0} \text{grad} \log \lambda, Y_j\right> 
- \left<\nabla_{X_0}(\text{grad} \log \lambda)_{E_0},Y_j\right>\\
&=\left< \nabla_{X_0} (\lambda^{-1}\text{grad}\,\lambda),Y_j\right> 
- \left<X_0, \text{grad} \log \lambda\right>\left<Y_j, -(\text{grad} \log \lambda)_{E_0^\perp}\right>\\
&= -\lambda^{-2}X_0(\lambda)Y_j(\lambda) 
+ \lambda^{-1}\text{Hess}\,\lambda(X_0,Y_j) + \lambda^{-2}X_0(\lambda)Y_j(\lambda)\\
&= \lambda^{-1}\text{Hess}\,\lambda(X_0,Y_j),
\end{align*}
and the statement follows.
\end{proof}  
  
  Given a net $\mathcal{F} = (F_i)_{i=0, \ldots, r}$ on a manifold $M$, 
a diffeomorphism $\Psi\colon \prod_{i=0}^r M_i \to M$ from a product manifold with product net 
$\mathcal{E} = (E_i)_{i=0, \ldots, r}$ is called a \emph{product representation} of $\mathcal{F}$ 
if $\Psi_*E_i(p) = F_i(\Psi(p))$ for $0 \leq i \leq r$. Combining Theorem 1 in \cite{rs} and 
Theorem~\ref{p:polarmetricumblicialnet} yields the following de Rham-type result.

\begin{theorem}\label{t:carspherical}
Let $M$ be a Riemannian manifold and let $\mathcal{E} = (E_i)_{i=0, \ldots, r}$ be an orthogonal net such that $E_a^\perp$ is totally umbilical for each $1 \leq a \leq r$.  
Then for every point $p \in M$ there exists a local product representation 
$\Psi\colon \prod_{i=0}^r M_i \to U$	of $\mathcal{E}$, with $p \in U \subset M$, which is conformal with respect to a polar metric on $\prod_{i=0}^r M_i$.  
\end{theorem}

The decomposition in Theorem \ref{t:carspherical} is of a local nature. 
Indeed, an example provided before Theorem 1 in \cite{pr} shows that a global representation can not always be achieved.

\begin{remark}\label{re:cpm}\emph{Observe that any orthogonal 
net $\mathcal{E} = (E_0, E_1)$ with only two factors such that $E_0$ has rank one satisfies the conditions in Thereom \ref{t:carspherical}, for any one-dimensional distribution is umbilical. Moreover, if $Z$ is a vector field  spanning $E_0$ on a simply connected open subset $U\subset M^n$, then the integrability of $E_1=E_0^\perp$ is equivalent to $Z$ being the gradient of a smooth function $\varphi\in C^{\infty}(U)$, the leaves of $E_1$ being the level sets of $\varphi$. Therefore, on any Riemannian manifold $M$  one can find as many such orthogonal nets $\mathcal{E}$, and hence as many local product representations of them that are conformal with respect to a polar metric, as smooth functions
on open subsets of $M$. 
}
\end{remark}

\section{Hypersurfaces of Enneper type }

 According to Theorem \ref{rpt},  hypersurfaces
$f\colon M^n= M_0^{n-1}\times I\to \mathbb{R}^{n+1}$ that are Ribaucour partial tubes over curves $\gamma\colon I\to \R^{n+1}$ 
are characterized by the fact that their product nets $\mathcal{E} = (E_0, E_1)$ satisfy conditions $(i)$ to $(iii)$ in Corollary \eqref{cor:rpt}, with $E_1$ of rank one. 
By Corollary \ref{cor:channel}, a special class of such hypersurfaces consists of channel hypersurfaces. 

Notice that, by Lemma \ref{prop:ext1}, conditions $(ii)$ and $(iii)$ in Corollary \eqref{cor:rpt} can be replaced by the requirement that the image by $f$ of each leaf $\sigma$ of $E_0$ be contained in a hypersphere of $\mathbb{R}^{n+1}$ that intersects $f(M)$ orthogonally along $f(\sigma)$.
 It is a natural problem to investigate the more general class of hypersurfaces  
$f\colon M^n= M_0^{n-1}\times I\to \mathbb{R}^{n+1}$ for which condition $(iii)$ is replaced by the following:
\begin{itemize}
\item[(iii')] The image by $f$ of each leaf of $E_0$ is contained in a hypersphere of $\mathbb{R}^{n+1}$ (which does not necessarily intersect $f(M)$ orthogonally along $f(\sigma)$).
\end{itemize}

  The next lemma shows that, under condition $(i)$, conditions $(ii)$ and $(iii')$ together are equivalent to requiring the image by $f$ of each leaf $\sigma$ of $E_0$ to be contained in a hypersphere of $\mathbb{R}^{n+1}$ that intersects $f(M)$ \emph{at a constant angle} along $f(\sigma)$.

\begin{lemma} \label{le:sphleaves} Let $f\colon M^n\to \mathbb{R}^{n+1}$ 
be a hypersurface and let $g\colon L^{n-1}\to M^n$ be  a hypersurface of $M^n$ such that $f(g(L))$ is contained in a hypersphere $\mathcal{S}$ of $\mathbb{R}^{n+1}$. Then $\mathcal{S}$ intersects $f(M)$ at a constant angle along $f(g(L))$ if and only if the shape operator $A$ of $f$ leaves $g_*TL$ invariant. 
 \end{lemma}
 \begin{proof} Let $P_0\in \R^{n+1}$ and $R>0$ be the center and the radius of $\mathcal{S}$,
 respectively, and let $\theta$ be the angle between its unit normal vector field $(f\circ g-P_0)/R$ and a unit normal vector field $N$ of $f$ along $f(g(L))$. Then
 $$
 f\circ g-P_0=R\cos\theta (N\circ g)+R\sin \theta f_*\delta,
 $$
 where $\delta$ is a unit normal vector field to $g$. Hence, for all $T\in \mathfrak{X}(L)$ we have
 $$
 f_*T=RT(\theta)(-\sin \theta N\circ g+\cos\theta f_*\delta)-R\cos\theta f_*Ag_*T+R\sin\theta(f_*\nabla_T\delta+\left<Ag_*T, \delta\right>N).
 $$
 Thus $T(\theta)=0$ for all $T\in \mathfrak{X}(L)$ if and only if $ \left<Ag_*T, \delta\right>=0$  for all $T\in \mathfrak{X}(L)$.
 \end{proof} 
 
 For $n=2$, conditions $(i)$ and $(ii)$ in Corollary \eqref{cor:rpt} mean that the leaves of $E_0$ and $E_1$ are lines of curvature of $f$. Condition $(iii')$ says that those correspondent to $E_0$ are contained in spheres.  Surfaces in $\R^3$ with these properties were called \emph{surfaces of Enneper type} in \cite{we}. Accordingly, we call a hypersurface $f\colon M^n= M_0^{n-1}\times I\to \mathbb{R}^{n+1}$, with product net $\mathcal{E} = (E_0, E_1)$, that satisfies conditions $(i)$, $(ii)$ and $(iii')$ above, a \emph{hypersurface of Enneper type}, or, more precisely, a hypersurface of Enneper type with respect to $\mathcal{E} = (E_0, E_1)$.

In the next subsection we first consider the case in which condition $(iii)$  is replaced by the following:
\begin{itemize}
\item[(iii'')] The image by $f$ of each leaf of $E_0$ is contained in an affine hyperplane of $\mathbb{R}^{n+1}$.
\end{itemize}    
 For $n=2$, surfaces that satisfy conditions $(i)$, $(ii)$ and $(iii'')$ are surfaces with planar lines of curvature correspondent to one of their principal curvatures. 
 
\subsection{Hypersurfaces of Enneper type with extrinsically planar leaves}   

   The next result shows how all hypersurfaces 
$f\colon M^n= M_0^{n-1}\times I\to \mathbb{R}^{n+1}$ that satisfy conditions $(i)$, $(ii)$ and $(iii'')$ can be constructed in terms of  Ribaucour partial tubes $N\colon M^n\to \mathbb{S}^{n}$ over curves in 
$\mathbb{S}^{n}$.

Notice that, for a Ribaucour partial tube $N\colon M^n\to \mathbb{S}^{n}$ over a curve in 
$\mathbb{S}^{n}$, the product net $\mathcal{E} = (E_0, E_1)$ of $M^n$ is a twisted product net with respect to the metric $d\sigma^2$ induced on $ M^n= M_0^{n-1}\times I$ by $N$,  that is, both $E_0$ and $E_1$ are umbilical distributions (with $E_0$ being, in fact, spherical), and 
$$d\sigma^2=v_0^2d\sigma_0^2+\nu^2ds^2,$$
 where $d\sigma_0^2$ is a metric
of constant curvature $1$ on $M_0$ and $ds^2$ is the standard metric on $I$. The mean curvature vector field of $E_0$ is 
$$H_0=-(\text{grad }\log v_0)_{E_1}=-\frac{1}{\nu^2}\frac{\partial (\log v_0)}{\partial s}\partial_s,$$
where $\partial_s$ is a unit vector field (with respect to the metric $ds^2$) along $I$.
Writing
$${\displaystyle \varphi=-\frac{1}{\nu}\frac{\partial (\log v_0)}{\partial s}},$$
that $E_0$ is spherical is equivalent to $\varphi$ depending only on $s$.

\begin{theorem} \label{thm:class} Let $N\colon M^n=M_0^{n-1}\times I\to \mathbb{S}^{n}$ be a 
Ribaucour partial tube over a unit-speed curve $\beta\colon I\to \mathbb{S}^{n}$.
Given $V\in C^{\infty}(I)$ and  $U\in C^{\infty}(M_0)$, define
$\gamma\in C^\infty(M)$ by
\begin{equation}\label{eq:gamma}
\gamma(x, s)=v_0(x,s)\left(U(x)+\int_0^{s}\frac{V(\tau)\nu(x,\tau)}{v_0(x,\tau)}d\tau\right).
\end{equation}
Then the map $f\colon M^n\to \mathbb{R}^{n+1}$ given  by
\begin{equation}\label{eq:gauss2}
f=\gamma (i\circ N)+i_*N_*\text{grad}\,\gamma,
\end{equation}
where $i\colon \mathbb{S}^{n}\to \mathbb{R}^{n+1}$ is the inclusion and $\text{grad}\,\gamma$ is computed with respect to the metric $d\sigma^2$ on $M^n$ induced by $N$, defines,
on the subset of its regular points,  a hypersurface satisfying 
conditions $(i)$, $(ii)$ and $(iii'')$.

Conversely, any  hypersurface $f\colon M_0^{n-1}\times I\to \mathbb{R}^{n+1}$ 
satisfying conditions $(i)$, $(ii)$ and $(iii'')$ whose shape operator has rank $n$ everywhere is given locally in this way.
\end{theorem}
\begin{proof} Let  $f\colon M^n\to \mathbb{R}^{n+1}$ be given  by
 \eqref{eq:gauss2} for some $\gamma\in C^{\infty}(M)$.  Differentiating \eqref{eq:gauss2} gives
\begin{equation}\label{eq:difer}
f_*=i_*N_*P,
\end{equation}
where $P=\text{Hess}\,\gamma+\gamma I$, the Hessian being computed 
with respect to $d\sigma^2$. 

Moreover, on the open subset where $P$ is invertible, that is, on the open subset of 
regular points of $f$, it follows from \eqref{eq:difer} that the map $N$ is the Gauss map of $f$ and that the shape operator of $f$ 
with respect to $N$ is 
\begin{equation}\label{eq:shape}
A=-P^{-1}.
\end{equation}

We claim that $\text{Hess}\,\gamma$
is adapted to  $\mathcal{E} = (E_0, E_1)$ if and only if $\gamma$ is given by \eqref{eq:gamma} for some $V\in C^{\infty}(I)$ and  $U\in C^{\infty}(M_0)$. 
Using that $E_0$ is umbilical with mean curvature vector field $H_0=\varphi\nu^{-1}\partial_s$ we obtain
\begin{align*}
\nabla_X \partial_s&=\nabla_X \nu(\nu^{-1}\partial_s)\\
&=X(\nu)\nu^{-1}\partial_s-\nu\|X\|^{-2}\left\langle\nabla_XX,\nu^{-1}\partial_s\right\rangle X\\
&=X(\log \nu)\partial_s-\nu\varphi X
\end{align*}
for all $X\in \Gamma(E_0)$. Hence
\begin{align*}
\text{Hess}\,\gamma(X, \partial_s)&=X\left(\frac{\partial \gamma}{\partial s}\right)-\left(\nabla_X \partial_s\right)(\gamma)\\
&=X\left(\frac{\partial \gamma}{\partial s}\right)-X(\log \nu)\frac{\partial \gamma}{\partial s}+\nu\varphi X(\gamma).
\end{align*}
Therefore $\text{Hess}\,\gamma(X, \partial_s)=0$ if and only if 
$$
X\left(\frac{\partial \gamma}{\partial s}\right)+\nu\varphi X(\gamma)=X(\log \nu)\frac{\partial \gamma}{\partial s}.
$$
Since $\varphi$ depends only on $s$, this can also be written as 
$$
X\left(\frac{\partial \gamma}{\partial s}+\varphi\nu\gamma\right)=X(\log \nu)\left(\frac{\partial \gamma}{\partial s}+\varphi\nu\gamma\right).
$$
Thus
${\displaystyle \frac{\partial \gamma}{\partial s}+\varphi\nu\gamma=\nu V}
$
for some $V\in C^{\infty}(I)$, which can be written as 
$$\frac{\partial (\gamma v_0^{-1})}{\partial s}=\nu Vv_0^{-1},$$
taking into account that ${\displaystyle \varphi\nu=-\frac{\partial (\log v_0)}{\partial s}}$. This proves our claim.

Notice that $\text{Hess}\,\gamma$
being adapted to  $\mathcal{E} = (E_0, E_1)$ implies both $P$ and $P^{-1}$ to be also 
adapted to  $\mathcal{E}$. This, together with \eqref{eq:difer}, \eqref{eq:shape} and the fact that
the product net $\mathcal{E} = (E_0, E_1)$ of $M^n$ is orthogonal with respect to 
the metric induced by $N$, implies that $\mathcal{E}$ is also orthogonal with  respect to 
the metric induced by $f$ and that the second 
fundamental form of $f$ is adapted to $\mathcal{E}$. 
Finally,  the image by $N$ of each leaf $\sigma$ of $E_0$ is a small hypersphere of $\mathbb{S}^{n}$.
Hence, if $\mathcal{H}$ is the hyperplane of $\mathbb{R}^{n+1}$ that is parallel to the 
affine hyperplane that contains $N(\sigma)$, 
then $N_*T\in \mathcal{H}$ for all $x\in \sigma$ and $T\in T_x\sigma$. 
It follows from \eqref{eq:difer} and the fact that 
$P$ leaves $E_0$ invariant that $f_*T\in \mathcal{H}$ 
for all $x\in \sigma$ and $T\in T_x\sigma$. 
Therefore $f(\sigma)$ is also contained in an affine hyperplane parallel to $\mathcal{H}$.

    Conversely, let $f\colon M^n= M_0^{n-1}\times I\to \mathbb{R}^{n+1}$ be a hypersurface 
satisfying conditions $(i)$, $(ii)$ and $(iii'')$ above and let $N\colon M^n\to \mathbb{S}^{n}$ be its Gauss map. 
If $\mathcal{H}$ is the hyperplane that is parallel to the affine hyperplane containing the image by $f$ of a leaf $\sigma$ of $E_0$, by condition $(ii)$ we have 
$$
N_*T=\tilde \nabla_T N=-f_*AT\in f_*T_x\sigma\in  \mathcal{H}
$$
for all $T\in T_x\sigma$. Hence $N(\sigma)$ is also contained in an affine hyperplane parallel to $\mathcal{H}$, 
and consequently it is an open subset of the small hypersphere of $\mathbb{S}^n$ given by its 
intersection with $\mathbb{S}^n$. Therefore, $N\colon M_0\times I\to \mathbb{S}^{n}$ is a local
diffeomorphism (by the assumption that the shape operator of $f$ has rank $n$ everywhere) with the property that the image by $N$ of any leaf of $E_0$ is a small hypersphere 
of $\mathbb{S}^{n}$. Thus $E_0$ is a spherical distribution with respect to the metric 
induced by $N$. Moreover, by condition $(ii)$ the images by $N$ of the integral curves of $E_1$
are orthogonal trajectories of the foliation of $\mathbb{S}^{n}$ given by the images of the leaves 
of $E_0$. In other words, the net $\mathcal{E}$ is also an orthogonal net with respect to the metric 
induced by $N$. It follows from Corollary \ref{rptdiffeo} that $N$ is a Ribaucour partial tube over a unit-speed curve $\beta\colon I\to \mathbb{S}^n$. 

  Now, the Gauss parametrization allows to recover $f$ in terms of $N$ and the support function $\gamma$ by means of \eqref{eq:gauss2}. Since the second fundamental form of $f$ is adapted to $\mathcal{E}$, then $\text{Hess}\,\gamma$ must also be adapted to $\mathcal{E}$ by \eqref{eq:shape}.  Thus $\gamma$ must be given by \eqref{eq:gamma}, as shown in the proof of the direct statement.
\end{proof}

For $n=2$, Theorem \ref{thm:class}  reads as follows.

\begin{corollary} \label{cor:planar} Let $N\colon J\times I\to \mathbb{S}^{2}$ be a Ribaucour partial tube over a 
unit-speed curve $\beta\colon I\to \mathbb{S}^{2}$ and let  $ds^2=v_1^2du_1^2+v_2^2du_2^2$ be the  metric induced by $N$. Given $U\in C^{\infty}(J)$ and $V\in C^{\infty}(I)$, let $\gamma\in C^\infty(J\times I)$ be given by
$$
\gamma(u_1, u_2)=v_1(u_1,u_2)\left(U(u_1)+\int_0^{u_2}\frac{V(\tau)v_2(u_1,\tau)}{v_1(u_1,\tau)}d\tau\right).
$$
Then the map $f\colon J\times I\to \mathbb{R}^{3}$ given by
$$
f(u_1,u_2)=\gamma(u_1,u_2)N(u_1,u_2)+\frac{1}{v_1^2}\frac{\partial \gamma}{\partial u_1}\frac{\partial N}{\partial u_1}
+\frac{1}{v_2^2}\frac{\partial \gamma}{\partial u_2}\frac{\partial N}{\partial u_2}
$$
defines, on the open subset of its regular points, a surface parametrized by lines of curvature whose
$u_1$-lines of curvature are planar. 

Conversely, any  surface in $\mathbb{R}^3$ free of flat points whose lines of curvature correspondent to one of its principal curvatures are planar can be locally parametrized in this way. 
\end{corollary}

\subsection{The general case} 

We now address the general problem of describing all hypersurfaces $f\colon M^n=M_0^{n-1}\times I\to \mathbb{R}^{n+1}$ of  Enneper type.

   If $N\colon M^n\to \mathbb{S}^n$ is the Gauss map of a hypersurface of Enneper type, it follows from conditions $(i)$ and $(ii)$ that the product net  $\mathcal{E} = (E_0, E_1)$  of $M^n$ is orthogonal also with respect to the metric $d\sigma^2$ induced by $N$. Hence, by Theorem~\ref{t:carspherical} (see also Remark \ref{re:cpm}), we can write
   \begin{equation}\label{eq:ds2}
   d\sigma^2=g_s+\rho^2ds^2
   \end{equation}
for some $\rho\in C^{\infty}(M)$ and for some smooth family of metrics $g_s$ on $M_0^{n-1}$
indexed on $I$. In particular, $\rho^{-1}\partial_s$ spans $E_1$ and has unit length  with respect to $d\sigma^2$.

 Let us assume that, for each $s_0\in I$,  the image by $f$ of the leaf $s=s_0$ of $E_0$ is contained in a hypersphere $\mathbb{S}^{n}(\gamma(s_0), R(s_0))$ of $\mathbb{R}^{n+1}$ with center $\gamma(s_0)\in \mathbb{R}^{n+1}$ and radius $R(s_0)$.  The position vector 
$\zeta(x,s)=f(x, s)-\gamma(s)$ of  $\mathbb{S}^{n}(\gamma(s), R(s))$ at $f(x,s)$
with respect to $\gamma(s)$ can be written as
$
\zeta=R\cos \theta N+R\sin \theta f_*\delta,
$
where $\delta$ is a unit vector field (with respect to the metric induced by $f$) spanning $E_1$ and $\theta\in C^{\infty}(M)$ is the angle between $\zeta$ and $N$, which depends only on $s$ by Lemma~\ref{le:sphleaves}.
After changing 
$\delta$ by $-\delta$, if necessary,  we can assume that $f_*\delta=\rho^{-1}N_*\partial_s$.  Thus we can write
\begin{equation}\label{eq:paramf}
f=\gamma+\alpha N +\beta \rho^{-1}N_*\partial_s,
\end{equation}
where $\alpha=\alpha(s)=R\cos\theta$ and $\beta=\beta(s)=R\sin \theta$.
Now we impose $N$ to be normal to $f$. Since $\alpha$ and $\beta$ depend only on $s$, the condition $0=\left<f_*T, N\right>$ is identically satisfied. On the other hand,
$0=\left<f_*\partial_s, N\right>$ if and only if
\begin{equation}\label{eq:condgab}
\left<\gamma', N\right>+\alpha'=\beta\rho,
\end{equation}
where the prime means derivative with respect to $s$. 
  We have thus proved the converse statement of the following result.

\begin{theorem}\label{prop:sphleaves0} Let $N\colon M^n\to \mathbb{S}^n$ be a local diffeomorphism of a product manifold $M^n= M_0^{n-1}\times I$ such that the product net 
 $\mathcal{E} = (E_0, E_1)$  of $M^n$ is orthogonal with respect to the metric $d\sigma^2$ induced by $N$ (equivalently, $N$ is a local diffeomorphism whose induced metric $d\sigma^2$ is given as in \eqref{eq:ds2} for some $\rho\in C^{\infty}(M)$ and for some smooth family of metrics $g_s$ on $M_0^{n-1}$
indexed on $I$). If there exist a smooth curve $\gamma\colon I\to \mathbb{R}^{n+1}$ and $\alpha, \beta\in C^{\infty}(I)$ such that \eqref{eq:condgab} holds, then the map $f\colon M^n\to \mathbb{R}^{n+1}$ given by \eqref{eq:paramf} parametrizes a hypersurface of Enneper type with respect to $\mathcal{E}$ having $N$ as a Gauss map. 

Conversely, if $f\colon M^n= M_0^{n-1}\times I\to \mathbb{R}^{n+1}$ is a hypersurface of Enneper type with respect to the product net $\mathcal{E}= (E_0, E_1)$ of $M^n$ having
$N\colon M^n\to \mathbb{S}^n$ as a Gauss map,
then there exist a smooth curve $\gamma\colon I\to \mathbb{R}^{n+1}$ and $\alpha, \beta\in C^{\infty}(I)$ satisfying \eqref{eq:condgab} such that  $f$ is given by \eqref{eq:paramf}.
\end{theorem}
\begin{proof} If  $f\colon M^n\to \mathbb{R}^{n+1}$ is given by \eqref{eq:paramf} in terms of $N$ and 
$(\gamma, \alpha, \beta)$, then \eqref{eq:condgab} is precisely the condition for $N$ to be a Gauss map of $f$. 
Moreover, if that condition is satisfied, then 
\begin{equation}\label{eq:radius}\|f-\gamma\|^2=\alpha^2 +\beta^2\,\,\,\mbox{and}\,\,\,
\left<\frac{f-\gamma}{\|f-\gamma\|}, N\right>=\frac{\alpha}{\sqrt{\alpha^2+\beta^2}}.
\end{equation}
The  preceding equations show that the image by $f$ of each leaf of $E_0$ is contained in a hypersphere of 
$\mathbb{R}^{n+1}$ that intersects $f(M)$ at a constant angle. By Lemma~\eqref{le:sphleaves}, the distribution $E_0$ is  invariant under the shape operator of $f$, that is,  the second fundamental form of $f$ is adapted to the net  $\mathcal{E} = (E_0, E_1)$. Since the net $\mathcal{E}$ is orthogonal with respect to the metric induced by $N$, this implies that it is also  orthogonal with respect to the metric induced by $f$. Thus $f$ is a hypersurface of Enneper type with respect to $\mathcal{E}$ having $N$ as a Gauss map. 
\end{proof}

   Given a hypersurface of Enneper type $f\colon M^n= M_0^{n-1}\times I\to \mathbb{R}^{n+1}$ with respect to the product net $\mathcal{E}= (E_0, E_1)$ of $M^n$,  we 
determine next all hypersurfaces of Enneper type  with respect to $\mathcal{E}$ sharing the same Gauss map with~$f$.

\begin{proposition}\label{prop:enneperfamily} Let  $f\colon M^n= M_0^{n-1}\times I\to \mathbb{R}^{n+1}$ be a hypersurface of Enneper type with respect to the product net $\mathcal{E}= (E_0, E_1)$ of $M^n$. Assume that its Gauss map $N\colon M^n\to \mathbb{S}^n$ is a local diffeomorphism whose induced metric $d\sigma^2$ is given by \eqref{eq:ds2} and that $f$ is parametrized by \eqref{eq:paramf} in terms of $N$, a smooth regular curve $\gamma\colon I\to \mathbb{R}^{n+1}$ and $\alpha, \beta\in C^{\infty}(I)$ satisfying \eqref{eq:condgab}. Suppose also that there does not exist any leaf $s=s_0$ of $E_0$ whose image by $f$ is 
(an open piece of) a round $(n-1)$-dimensional sphere. Then any other hypersurface 
$\bar f\colon M^n\to \mathbb{R}^{n+1}$ of Enneper type with respect to $\mathcal{E}$ having $N$ as a Gauss map 
is parametrized  by \eqref{eq:paramf} in terms of a smooth curve $\bar\gamma\colon I\to \mathbb{R}^{n+1}$ and $\bar\alpha, \bar\beta\in C^{\infty}(I)$ which are  related to
$\gamma$, $\alpha$ and $\beta$ by
$$
\bar\beta=\lambda \beta,\,\,\,\bar\alpha'=\lambda \alpha'\,\,\,\mbox{and}
\,\,\,\bar\gamma'=\lambda\gamma'
$$
for some $\lambda\in C^{\infty}(I)$.
\end{proposition}
\begin{proof} Let $\bar f\colon M^n\to \mathbb{R}^{n+1}$ be another hypersurface of Enneper type with respect to $\mathcal{E}$ sharing the same Gauss map $N$ with $f$. By Theorem \ref{prop:sphleaves0},  it can be  parametrized  by \eqref{eq:paramf} in terms of a smooth curve $\bar\gamma\colon I\to \mathbb{R}^{n+1}$ and $\bar\alpha, \bar\beta\in C^{\infty}(I)$ satisfying 
 \begin{equation}\label{eq:cond2}
 \left<\bar\gamma', N\right>+\bar\alpha'=\bar\beta\rho.
 \end{equation}
Notice that if $\beta$ vanishes at some $s_0\in I$, since $\gamma'$ is nowhere vanishing by assumption then \eqref{eq:condgab} implies that the image by $N$ of the leaf $s=s_0$ of $E_0$
is contained in an affine hyperplane of $\mathbb{R}^{n+1}$ (hence is a $(n-1)$-dimensional round hypersphere  of $\mathbb{S}^n$). Hence also the 
image by $f$ of the leaf $s=s_0$ of $E_0$
is contained in an affine hyperplane of $\mathbb{R}^{n+1}$, and therefore it is an open piece of an $(n-1)$-dimensional round hypersphere given by its intersection with the hypersphere 
 $\mathbb{S}^{n}(\gamma(s_0), R(s_0))$ containing such image, contradicting our assumption.   Thus $\beta$ is nowhere vanishing and we can write
 $\bar\beta=\lambda\beta$ for some $\lambda\in C^{\infty}(I)$.  Comparing \eqref{eq:cond2} with  \eqref{eq:condgab} yields
\begin{equation}\label{eq:comp}
 \left<\bar\gamma' -\lambda \gamma', N\right>+\bar\alpha'-\lambda\alpha'=0.
\end{equation}
If $\bar\gamma' -\lambda \gamma'$ was nonzero for some $s_0\in I$, then arguing as before we would conclude that the image by $f$ of the leaf $s=s_0$ of $E_0$ would be an open piece of an $(n-1)$-dimensional round hypersphere, a contradiction.  Thus $\bar\gamma' -\lambda \gamma'$, and hence also $\bar\alpha' -\lambda \alpha'$ by \eqref{eq:comp}, must vanish everywhere.
\end{proof}

We are now able to show how any hypersurface of Enneper type  in $\mathbb{R}^{n+1}$ satisfying the assumptions of Proposition \ref{prop:enneperfamily} can be constructed
by means of a hypersurface of Enneper type with extrinsically planar leaves.

\begin{theorem}\label{thm:sphleaves} Let $f\colon M^n= M_0^{n-1}\times I\to 
\mathbb{R}^{n+1}$ be a hypersurface given as in Theorem \ref{thm:class}, let  
$\tilde f={\mathcal I}\circ f\colon M^n= M_0^{n-1}\times I\to 
\mathbb{R}^{n+1}$ be its composition with an inversion with respect to a hypersphere of unit radius centered at the origin and let $\tilde N \colon M^n\to \mathbb{S}^{n}$ be the Gauss map of $\tilde f$.  Let $\tilde f$ be 
parametrized by \eqref{eq:paramf} in terms of $\tilde N$ and a triple $(\tilde\gamma, \tilde\alpha, \tilde\beta)$  satisfying \eqref{eq:condgab}, where $\tilde\alpha, \tilde\beta\in C^{\infty}(I)$ and $\tilde\gamma\colon I\to \mathbb{R}^{n+1}$ is a smooth curve. Define a new triple  
$(\bar\gamma, \bar\alpha, \bar\beta)$ by  
\begin{equation}\label{newtriple}
\bar\beta=\lambda \tilde\beta,\,\,\,\bar\alpha'=\lambda \tilde\alpha'\,\,\,\mbox{and}\,\,\,\bar\gamma'=\lambda\tilde\gamma'
\end{equation}
for some  $\lambda\in C^{\infty}(I)$.
Then the hypersurface  $\bar f\colon M^n= M_0^{n-1}\times I\to 
\mathbb{R}^{n+1}$ parametrized by \eqref{eq:paramf} in terms of $\tilde N$ and 
$(\bar\gamma, \bar\alpha, \bar\beta)$ 
is of Enneper type.

Conversely, any  hypersurface of Enneper type  in $\mathbb{R}^{n+1}$ satisfying the assumptions of Proposition \ref{prop:enneperfamily} can be constructed as above.
\end{theorem}
\begin{proof} Since $f$ satisfies conditions $(i)$, $(ii)$ and $(iii'')$ above, it is clear that $\tilde f={\mathcal I}\circ f$ satisfies conditions $(i)$, $(ii)$ and $(iii')$, and hence is a hypersurface of Enneper type. If the triple  $(\tilde\gamma, \tilde\alpha, \tilde\beta)$  satisfies \eqref{eq:condgab}, then the same holds for the new triple 
$(\bar\gamma, \bar\alpha, \bar\beta)$ defined by  \eqref{newtriple}. 
Then the hypersurface  $\bar f\colon M^n= M_0^{n-1}\times I\to 
\mathbb{R}^{n+1}$ parametrized by \eqref{eq:paramf} in terms of $\tilde N$ and 
$(\bar\gamma, \bar\alpha, \bar\beta)$ 
is of Enneper type by Theorem \ref{prop:sphleaves0}. 

To prove the converse, let  $\bar f\colon M^n= M_0^{n-1}\times I\to 
\mathbb{R}^{n+1}$ be a hypersurface of Enneper type with respect to the product net  $\mathcal{E} = (E_0, E_1)$  of $M^n$, parametrized by \eqref{eq:paramf} in terms of its Gauss map $\bar N\colon M^n\to \mathbb{S}^n$, a smooth regular curve $\bar\gamma\colon I\to \mathbb{R}^{n+1}$ and $\bar\alpha, \bar\beta\in C^{\infty}(I)$ satisfying \eqref{eq:condgab}. Given $\lambda\in C^{\infty}(I)$,  let $\tilde\gamma\colon I\to \mathbb{R}^{n+1}$ and $\tilde\alpha, \tilde\beta\in C^{\infty}(I)$ be related to $\bar \gamma, \bar\alpha, \bar\beta$ by 
$$
\lambda\bar\beta= \tilde\beta,\,\,\,\lambda\bar\alpha'=\tilde\alpha',\,\,\,\lambda\bar\gamma'=\tilde\gamma'.
$$
One can choose $\lambda$ so that  
\begin{equation}\label{eq:tilde}
\|\tilde\gamma\|^2=\tilde\alpha^2+\tilde\beta^2.
\end{equation}
In fact, this amounts to choosing $\lambda$ as one of the infinitely many solutions of the equation
$$
\|\int \lambda\bar\gamma'\|^2=\lambda \bar\beta^2+(\int \lambda\bar\alpha')^2.
$$
 Let $\tilde f \colon M^n\to 
\mathbb{R}^{n+1}$ be the hypersurface  given by \eqref{eq:paramf} in terms of $\bar N$ and 
$(\tilde\gamma, \tilde\alpha, \tilde\beta)$. By Theorem \ref{prop:sphleaves0}, $\tilde f$ also satisfies conditions $(i)$, $(ii)$ and $(iii')$ and has $\bar N$ as a Gauss map. Moreover, Eqs. \eqref{eq:radius} and \eqref{eq:tilde} imply that the hyperspheres that contain the images by $\tilde f$ of the leaves of $E_0$ all pass through the origin. Therefore, the composition $f=\mathcal{I}\circ \tilde f$ of $\tilde f$ with an inversion with respect to a hypersphere of unit radius centered at the origin  satisfies conditions $(i)$, $(ii)$ and $(iii'')$, and hence is given as in Theorem~\ref{thm:class}.
\end{proof}

\subsection{On some special hypersurfaces of Enneper type}

  In this subsection we discuss some special hypersurfaces $f\colon M^n= M_0^{n-1}\times I\to 
\mathbb{R}^{n+1}$ of Enneper type with respect to the product net  $\mathcal{E} = (E_0, E_1)$  of $M^n$, in particular those with the property that the hyperspheres containing the images by $f$ of the leaves of $E_0$ are concentric, which are ruled out in the converse statement of Theorem~\ref{thm:sphleaves}. 
  
   First recall that there exists a conformal diffeomorphism $\Phi\colon  \mathbb{S}^{n-1}\times\mathbb{R} \to \mathbb{R}^n \setminus\{0\}$ given by $(x,t) \mapsto e^t x$.
Similarly, there is also a  conformal diffeomorphism between
$\mathbb{H}^n \times \mathbb{S}^1$ and $\mathbb{R}^{n+1} \setminus \mathbb{R}^{n-1}$ given as follows. Let $e_0, e_1, \ldots, e_{n-1},e_n$ be a pseudo-orthonormal basis of the 
Lorentzian space $\R^{n+1}_1$ satisfying $\left<e_0,e_0\right>=0=\left<e_n, e_n\right>$, 
$\left<e_0, e_n\right>=-1/2$ and $\left<e_i,e_j\right>
=\delta_{ij}$ for $1\leq i\leq n-1$ and $0\leq j\leq n$.
 Then the map 
$
\bar\Phi \colon \mathbb{H}^n\times \mathbb{S}^{1}\subset \mathbb{R}^{n+1}_1 
\times \mathbb{R}^{2} \to \mathbb{R}^{n+1}\setminus \mathbb{R}^{n-1}
$
given by
\begin{equation}\label{eq:barphi}
\bar\Phi(x_0 e_0 + \ldots + x_{n} e_n, (y_1, y_2)) = \frac{1}{x_0}(x_1, 
\ldots, x_{n-1}, y_1, y_2)
\end{equation}
is a conformal diffeomorphism. Composing $\bar\Phi$ with the isometric covering map 
$$
\pi\colon \mathbb{H}^n\times \mathbb{R}\to \mathbb{H}^n\times \mathbb{S}^1: (x,t) \mapsto (x, (\cos t, \sin t))
$$ 
gives rise to a conformal covering map 
$ \Phi \colon \mathbb{H}^n\times \mathbb{R}\to \mathbb{R}^{n+1}\setminus \mathbb{R}^{n-1}$
given by
\begin{equation}\label{tpsi}
\Phi(x_0 e_0 + \ldots + x_n e_n, t) = 
\frac{1}{x_0}(x_1, \ldots, x_{n-1}, \cos t, \sin t).
\end{equation}

  Now let   $g\colon\, M^{n-1}\to \mathbb{Q}_\epsilon^n$  be a hypersurface, where $\mathbb{Q}_\epsilon^n$ stands for $\mathbb{S}^n$ if $\epsilon=1$, $\mathbb{R}^n$ if $\epsilon=0$  and  $\mathbb{H}^n$ if $\epsilon=-1$,   and let $g_s \colon\, M^{n-1}\to \mathbb{Q}_\epsilon^n\subset \R_\mu^{n+|\epsilon|}$, where $\mu=0$ if $\epsilon=0$ or $\epsilon=1$, and $\mu=1$ if $\epsilon=-1$, be the family of its parallel hypersurfaces, that is, 
$$g_s(x)=C_\epsilon(s)g(x)+S_\epsilon(s)N(x),$$
where $N$ is a unit normal vector field to $g$, 
$$
C_\epsilon(s)=\left\{\begin{array}{l}
\cos s, \,\,\,\mbox{if}\,\,\epsilon=1
\vspace{1.5ex}\\
1, \,\,\,\mbox{if}\,\,\epsilon=0
\vspace{1.5ex}\\
\cosh s, \,\,\,\mbox{if}\,\,\epsilon=-1
\end{array}\right.\,\,\,\,\,\,\,\,\mbox{and}\,\,\,\,\,\,\,\,\,\,\,
S_\epsilon(s)=\left\{\begin{array}{l}
\sin s, \,\,\,\mbox{if}\,\,\epsilon=1
\vspace{1.5ex}\\
s, \,\,\,\mbox{if}\,\,\epsilon=0
\vspace{1.5ex}\\
\sinh s, \,\,\,\mbox{if}\,\,\epsilon=-1.
\end{array}\right.
$$

 Define 
$$F\colon M^n=M^{n-1}\times I\to \mathbb{Q}_\epsilon^{n}\times \R\subset \R_\mu^{n+1+|\epsilon|}$$  by
\begin{equation}\label{eq:constantangle}F(x,s)=g_s(x)+a(s)\frac{\partial}{\partial t}
\end{equation}
for some smooth function $a\colon\, I\to \R$ with positive derivative on an open interval $I\subset \R$. Here, we regard $F$ as taking values in the underlying flat space.  
 
 In the next statement, we denote by $\Phi$  either the conformal diffeomorphism 
 $\Phi\colon \mathbb{Q}_\epsilon^n\times \R\to \R^{n+1}\setminus \{0\}$ if $\epsilon=1$, the conformal covering map $\Phi\colon \mathbb{Q}_{\epsilon}^n\times \R\to \R^{n+1}\setminus \R^{n-1}$ if $\epsilon=-1$ or the standard isometry $\Phi\colon \mathbb{Q}_\epsilon^n\times \R\to  \R^{n+1}$ if $\epsilon=0$. 
 
 \begin{theorem} \label{thm:joach} Let $f\colon M^n=M^{n-1}\times I\to \mathbb{R}^{n+1}$ be a hypersurface of Enneper type with respect to the product net $\mathcal{E} = (E_0, E_1)$ of $M^n$. Assume that the images by $f$ of the leaves of $E_0$ are contained in either
 \begin{itemize}
 \item[(a)] concentric hyperspheres.
 \item[(b)]  parallel affine hyperplanes.
 \item[(c)] affine hyperplanes  intersecting along an affine $(n-1)$-dimensional subspace.
 \end{itemize}
 Then, assuming $M^n$ simply connected in case $(c)$, there exists   $F\colon M^n\to \mathbb{Q}_\epsilon^n\times \R$  given by (\ref{eq:constantangle}) in terms of a hypersurface $g\colon\, M^{n-1}\to \mathbb{Q}_\epsilon^n$, with $\epsilon=1$ in case $(a)$,  $\epsilon=0$ in case $(b)$ and  $\epsilon=-1$ in case $(c)$, such that $f=\Phi \circ F$. 
\end{theorem}
\begin{proof} Let $ \Psi\colon  \mathbb{R}^{n+1}\setminus \mathbb{R}^{k-1} \to 
\mathbb{H}^k \times \mathbb{S}^{n-k+1} \subset \R^{k+1}_1 \times \R^{n-k+2}=\mathbb{R}_1^{n+3}$ be  the map  given by  
$$ \Psi(y_1, \ldots, y_{n+1}) = 
\frac{1}{\sqrt{y_k^2+\cdots +y_{n+1}^2}} \left( e_0 + \sum_{i=1}^{k-1} y_ie_i 
+ \left( \sum_{i=1}^{n+1} y_i^2 \right) e_k, (y_k, \ldots, y_{n+1}) \right), $$
where $\mathbb{R}^{k-1}=\{(y_1, \ldots, y_{n+1}) \in \R^{n+1}\,:\, y_k=\cdots =y_{n+1}=0\}$ and $e_0, \ldots, , e_{k}$ is a 
pseudo-orthonormal basis of $\R_1^{k+1}$ with $\left<e_0, e_0\right>=0=\left<e_{k}, e_{k}\right>$, $\left<e_0, e_{k}\right>=-1/2$ and $\left<e_i,e_j\right>=\delta_{ij}$ for $1\leq i\leq k-1$ and $0\leq j\leq k$.  It is a conformal diffeomorphism whose conformal factor $\varphi\in C^{\infty}(\mathbb{R}^{n+1}\setminus \mathbb{R}^{k-1})$ is 
$\varphi(y_1,  \ldots, y_{n+1})=(\sum_{i=k}^{n+1} y_i^2)^{1/2}$. For $k=n$,  $\Psi$ is the inverse of the conformal diffeomorphism 
$\bar\Phi \colon \mathbb{H}^n\times \mathbb{S}^{1} \to \mathbb{R}^{n+1}\setminus 
\mathbb{R}^{n-1}$ given by \eqref{eq:barphi}.  Notice that $\Psi$ takes each half-space of a $k$-dimensional subspace $\R^k$ containing $\R^{k-1}$ onto a slice 
$\mathbb{H}^k\times \{x\}$ of 
$\mathbb{H}^k \times \mathbb{S}^{n-k+1}$, while $(n-k+1)$-dimensional spheres  centered at
 $\R^{k-1}$ lying in subspaces $\R^{n-k+2}$ orthogonal to $\R^{k-1}$ are mapped onto slices
  $\{x\}\times \mathbb{S}^{n-k+1}$ of $\mathbb{H}^k \times \mathbb{S}^{n-k+1}$. 
  
  Assume first that condition $(a)$ is satisfied. Let $\Psi$ be the diffeomorphism defined above for $k=1$, and let $\hat \Psi$ be its composition with the isometry
  $$
  (y^{-1}e_0+ye_1, x)\in \mathbb{H}^1\times \mathbb{S}^{n}\; (y>0)\mapsto
  (x, \log y)\in \mathbb{S}^{n}\times \R. 
  $$ 
We can assume that the hyperspheres containing  the images by $f$ of the leaves of $E_0$ are centered at the origin. Then the images by $F=\hat\Psi\circ f\colon M^n\to \mathbb{S}^n\times \R$ of the leaves of $E_0$ are contained in the slices 
   $\mathbb{S}^{n}\times \{t\}$, $t\in \R$. This means that the height function
   $(x,s) \mapsto \left<F(x,s),\partial_t\right>$ depends only on $s$, where $\partial_t$ is a unit vector field tangent to the factor  $\R$. Differentiating with respect to $X\in E_0$ gives
   $0=\left<F_*X, \partial_t\right>=\left<F_*X, F_*T\right>$, where $\partial_t=F_*T+\left<\partial_t, N\right>N$. Since $\hat\Psi$ is a conformal diffeomorphism, the metrics induced by $f$ and $F$ are conformal, hence $T$ spans $E_1$. Moreover, since conformal diffeomorphisms preserve principal directions and the integral curves of $E_1$ are lines of curvature of $f$, it follows that $T$ is a principal direction of $F$.

  Now assume that condition $(b)$ holds. Denoting by $\Psi\colon \R^{n+1}\to \mathbb{R}^n\times \R\to \R^{n+1}$ the standard isometry, the images under $F=\Psi\circ f$ of the leaves of $E_0$ are contained in slices $\mathbb{R}^n\times \{t\}$, $t\in \R$. Arguing as in the preceding paragraph, this means that the tangent component of the unit vector field tangent to the factor $\R$ is a principal direction of $F$.
  
  Finally,  suppose that condition $(c)$ holds and let $\Psi$ be the diffeomorphism defined in the preceding paragraph for $k=n$. Then the images by $\hat F=\Psi\circ f\colon M^n\to \mathbb{H}^n\times \mathbb{S}^1$ of the leaves of $E_0$ are contained in slices 
   $\mathbb{H}^{n}\times \{x\}$, $x\in \mathbb{S}^1$. Let $F\colon M^n\to \mathbb{H}^n\times \mathbb{R}$ be such that $\hat F=\pi\circ F$, where $\pi\colon \mathbb{H}^n\times \mathbb{R}\to \mathbb{H}^n\times \mathbb{S}^1$ is the covering map $(x, t)\mapsto (x, (\cos t, \sin t))$. Then the images by $F$ of the leaves of $E_0$ are contained in slices 
   $\mathbb{H}^{n}\times \{t\}$, $t\in \mathbb{R}$. Arguing as in case $(a)$ we conclude that $T$ is a principal direction of $F$, where $\partial_t=F_*T+\left<\partial_t, N\right>N$, with $\partial_t$  a unit vector field tangent to $\R$.
   
   In either of the preceding cases, it follows from Theorem $1$ in \cite{t5} that the map $F$  is given by (\ref{eq:constantangle}) in terms of a hypersurface $g\colon\, M^{n-1}\to \mathbb{Q}_\epsilon^n$, with $\epsilon=1$ in case $(a)$,  $\epsilon=0$ in case $(b)$ and  $\epsilon=-1$ in case $(c)$. In either case we have $f=\Phi\circ F$, thus the statement follows.
   \end{proof}
   
   Surfaces of Enneper type in $\R^3$ satisfying condition $(c)$ in the preceding theorem are known in the literature as \emph{Joachimsthal surfaces}. These are the surfaces whose lines of curvature correspondent to one of the principal curvatures are contained in planes that intersect along a common line, whereas 
the lines of curvature correspondent to the other principal curvature lie on spheres centered on that line. The following consequence of Theorem \ref{thm:joach} shows how any such surface arises.

 \begin{corollary} \label{cor:joach} Let   $\gamma\colon J\subset \R\to \mathbb{H}^2$  be a
 unit-speed curve and let $\gamma_s \colon J\subset \R\to \mathbb{H}^2\subset \R_1^3$ be the family of its parallel curves, that is, 
$$\gamma_s(t)=\cosh(s)\gamma(t)+\sinh(s)\gamma(t)\wedge \gamma'(t).$$
 Define 
$F\colon J\times I\to \mathbb{H}^{2}\times \R\subset \R_1^4$  by
$$
F(t,s)=\gamma_s(t)+a(s)\frac{\partial}{\partial t},
$$
where  $I\subset \R$ is an open interval and $a\in C^{\infty}(I)$ has positive derivative. Then, on the subset $M^2\subset J\times I$ of its regular points, the map $f=\Phi\circ F\colon M^2\to \R^3\setminus \R$, where $\Phi\colon \mathbb{H}^{2}\times \R \to \R^3\setminus \R$ is the conformal covering map given by \eqref{tpsi}, defines a Joachimsthal surface.

   Conversely, any  Joachimsthal surface in $\R^3$ can be parametrized in this way.
\end{corollary}

\section{A decomposition theorem}

  The aim of this last section is to prove the following decomposition
theorem for immersions of product manifolds.  

\begin{theorem}\label{thm:decomp}
Let $f\colon M = \prod_{i=0}^r M_i \to \R^m$ be a conformal immersion with conformal 
factor $\lambda\in C^{\infty}(M)$ of a product manifold $M$ 
endowed with a polar metric. Assume that the second fundamental form of $f$ is adapted to 
the product net $\mathcal{E} = (E_i)_{i=0, \cdots, r}$ of $M$. 
If $\text{dim}\,M_0=1$, suppose further that $\text{Hess}\,\lambda$ is adapted
to the net $(E_0, E_0^\perp)$. Then $f$ is a Ribaucour partial tube over 
an immersion $\tilde{f}\colon \tilde{M} = \prod_{a=1}^r M_a \to \R^m$ given in one of the following ways:
\begin{enumerate}
\item $\tilde f = \tau \circ \bar{f}$, where  $\bar{f}$ is an extrinsic product of substantial immersions 
into either $\R^m$ or $\mathbb{S}^m_c$ and $\tau$ is, respectively, a conformal transformation of $\R^m$
or a conformal diffeomorphism of $\mathbb{S}_c^m$ (with one point removed) onto $\R^m$. 
\item $\tilde f = \Theta \circ (f_1 \times \bar{f})$, where 
$\Theta\colon \mathbb{H}_{-c}^k \times \mathbb{S}_{c}^{m-k} \to \mathbb{R}^m$ is a conformal diffeomorphism (onto the complement of a $(k-1)$-dimensional subspace),
and, after possibly relabelling factors, $f_1\colon M_1 \to \mathbb{H}_{-c}^k$ is an isometric immersion
and $\bar{f}\colon \prod_{a=2}^r M_i \to \mathbb{S}_{c}^{m-k}$ is an extrinsic product of substantial isometric 
immersions.
\end{enumerate} 
\end{theorem}

\begin{proof}
First, suppose that $r=1$. In this case, by Proposition \ref{p:polarmetricumblicialnet}, 
the assumption that the metric induced by $f$ is conformal to a polar metric says that 
$E_0$ is an umbilical distribution.
It follows from Lemma $12$ of \cite{t2} if $\text{dim}\,M_0\geq 2$, or the assumption that 
$\text{Hess}\,\lambda$ is adapted to the net $(E_0, E_1)$ combined with Proposition~\ref{prop:hess} 
if $\text{dim}\,M_0=1$, that $E_0$ is indeed spherical. 
Since the second fundamental form of $f$ is adapted with respect to the product net of $M$,
the statement in this case is a consequence of Theorem \ref{rpt}.

  If $r\geq 2$ is arbitrary, apply the case $r=1$ just proved to $f$ regarded as an immersion of 
$M_0 \times \tilde{M}$ into $\mathbb{R}^m$, where $\tilde{M}= \prod_{a=1}^r M_i$. 
It follows that $f$ is a Ribaucour partial tube over  
$\tilde{f}\colon \tilde{M}\to \R^m$ given by $\tilde{f}=f\circ \mu_{\bar{x}_0}$ for some
$\bar{x}_0\in M_0$, where $\mu_{\bar{x}_0}\colon \tilde{M}\to M$ is the inclusion given by 
$\mu_{\bar{x}_0}(\tilde{x})=(\bar{x}_0,\tilde{x})$.
The  metric $\tilde{g}$ induced on $\tilde{M}$ by $\mu_{\bar{x}_0}$ from the metric $g$ of $M$ is 
\begin{align*}
\tilde{g} &= \mu_{\bar{x}_0}^*g= (\lambda \circ \mu_{\bar{x}_0})\mu_{\bar{x}_0}^* \big( \pi_0^*g_0 
+ \sum_{a=1}^r \pi_a^*(g_a \circ \pi_0) \big)= (\lambda \circ \mu_{\bar{x}_0})\sum_{a=1}^r \tilde{\pi}_a^*g_a(\bar{x}_0),
\end{align*}
where $\tilde{\pi}_a\colon \tilde{M} \to M_a$ is the projection.  
Hence $\tilde{g}$ is conformal to a Riemannian product metric.  

We now claim that the second fundamental form of  $\tilde{f} = f \circ \mu_{\bar{x}_0}$ is adapted to the product net on $\tilde{M}$.  
We have
\begin{align*}
\alpha^{\tilde{f}}(\tau_a^{\tilde{x}}{}_*X_a,\tau_b^{\tilde{x}}{}_*X_b) &= \alpha^f(\tau_a^{(\bar{x}_0,\tilde{x})}{}_*X_a,\tau_b^{(\bar{x}_0,\tilde{x})}{}_*X_b)  + f_*\alpha^\mu_{\bar{x}_0}(\tau_a^{\tilde{x}}{}_*X_a,\tau_b^{\tilde{x}}{}_*X_b)\\
&=  f_*\alpha^\mu_{\bar{x}_0}(\tau_a^{\tilde{x}}{}_*X_a,\tau_b^{\tilde{x}}{}_*X_b),	
\end{align*}
since the second fundamental form of $f$ is adapted to the product net.  
Because $E_b^\perp$ is totally umbilic,
$$\left<\nabla_{\tau_a^{\tilde{x}}{}_*X_a} \tau_b^{\tilde{x}}{}_*X_b, \tau_0^{\tilde{x}}{}_*X_0\right> 
= - \left< \tau_b^{\tilde{x}}{}_*X_b, \nabla_{\tau_a^{\tilde{x}}{}_*X_a} \tau_0^{\tilde{x}}{}_*X_0\right> = 0. $$
Thus $\alpha^\mu_{\bar{x}_0}(\tau_a^{\tilde{x}}{}_*X_a,\tau_b^{\tilde{x}}{}_*X_b) = 0$, and our claim follows. 

    To complete the proof, it remains to show that $\tilde{f}\colon \prod_{a=1}^r M_i\to \R^m$ is given as in the statement. But this follows from Theorem 5 in\cite{t2},
where conformal immersions of a Riemannian product whose second fundamental forms are adapted to the product net of the manifold have been classified.
\end{proof}

\begin{corollary}\label{cor:decomp}
Let $f\colon M^m = \prod_{i=0}^r M_i \to \R^m$ be a conformal local diffeomorphism with conformal 
factor $\lambda\in C^{\infty}(M)$ of a product manifold $M$ 
endowed with a polar metric. If $n_0=\text{dim}\,M_0=1$, suppose further that $\text{Hess}\,\lambda$ is 
adapted to the net $(E_0, E_0^\perp)$. Then $f$ is a Ribaucour partial tube over 
an immersion $\tilde{f}\colon \tilde{M} = \prod_{a=1}^r M_a \to \R^m$ given as in Theorem \ref{thm:decomp},
with each of the immersions $f_a$, $1\leq a\leq r$, having flat normal bundle and
 $f_0\colon M_0^{m_0}\to \mathbb{R}^{m_0}$ being a local isometry.   
\end{corollary}

{\renewcommand{\baselinestretch}{1}
\hspace*{-30ex}\begin{tabbing}
\indent \= Universidad de Murcia \hspace{18.2ex} 
           Universidade de S\~ao Paulo \\
\>         Departamento de Matematicas \hspace{10ex}
           Av. Trabalhador S\~ao-Carlense 400 \\
\>         E-30100 Espinardo \hspace{22.2ex} 
           13560-970 --- S\~ao Carlos  \\
\>         Murcia, Spain\hspace{28ex} 
           Brazil\\
\>         sjchiona@gmail.com  \hspace{21.3ex}
           tojeiro@icmc.usp.br
\end{tabbing}}
\end{document}